\documentclass{amsart}

\usepackage{amsmath, amssymb, amsthm, mathrsfs, graphicx}
\usepackage{hyperref, enumitem, xspace, ifthen}
\usepackage[all]{xy}
\usepackage{tikz}
\usetikzlibrary{arrows,shapes,trees}

\newtheorem*{thm-plain}{Theorem}
\newtheorem{thm}{Theorem}[section]
\newtheorem{lem}[thm]{Lemma}
\newtheorem{prp}[thm]{Proposition}
\newtheorem{cor}[thm]{Corollary}

\newtheorem{conj}[thm]{Conjecture}

\numberwithin{equation}{thm}

\theoremstyle{definition}
\newtheorem{dfn}[thm]{Definition}
\newtheorem{prpdfn}[thm]{Proposition/Definition}
\newtheorem*{dfn-plain}{Definition}

\theoremstyle{remark}
\newtheorem{clm}[thm]{Claim}
\newtheorem{obs}[thm]{Observation}
\newtheorem{awlog}[thm]{Additional Assumption}
\newtheorem{ntn}[thm]{Notation}
\newtheorem{rem}[thm]{Remark}
\newtheorem{warn}[thm]{Warning}
\newtheorem{exm}[thm]{Example}
\newtheorem*{rem-plain}{Remark}

\DeclareMathOperator{\img}{im}

\DeclareMathOperator{\Pic}{Pic}
\DeclareMathOperator{\codim}{codim}
\DeclareMathOperator{\Tor}{tor}

\DeclareMathOperator{\Exc}{Exc}
\DeclareMathOperator{\res}{res}
\DeclareMathOperator{\restr}{restr}
\DeclareMathOperator{\mult}{mult}
\DeclareMathOperator{\discrep}{discrep}
\DeclareMathOperator{\coker}{coker}

\def\rd#1.{\lfloor{#1}\rfloor}
\def\rp#1.{\lceil{#1}\rceil}

\newcommand{\N}{\mathbb N}

\newcommand{\Q}{\ensuremath{\mathbb Q}}
\newcommand{\R}{\mathbb R}
\newcommand{\C}{\mathbb C}

\renewcommand{\P}{\mathbb P}
\renewcommand{\O}{\mathcal O}

\newcommand{\x}{\times}

\renewcommand{\phi}{\varphi}
\renewcommand{\theta}{\vartheta}
\newcommand{\id}{\mathrm{id}}

\newcommand{\minus}{\backslash}
\newcommand{\inj}{\hookrightarrow}

\newcommand{\isom}{\cong}

\newcommand{\mc}{\mathcal}

\newcommand{\ms}{\mathscr}
\newcommand{\tensor}{\otimes}
\newcommand{\Hnought}{\mathit\Gamma}
\newcommand{\Hom}{\mathrm{Hom}}
\newcommand{\sHom}{\mathscr H\!om}

\newcommand{\sg}{\mathrm{sg}}
\newcommand{\sm}{\mathrm{sm}}
\newcommand{\snc}{\mathrm{snc}}
\newcommand{\nsnc}{\mathrm{nsnc}}

\newcommand{\Sym}{\mathrm{Sym}}

\newcommand{\CC}{\ensuremath{\mathcal C}}
\newcommand{\Diff}{\mathrm{Diff}}

\newcommand{\sA}{\mathscr{A}}
\newcommand{\sB}{\mathscr{B}}

\newcommand{\sE}{\mathscr{E}}
\newcommand{\sF}{\mathscr{F}}

\newcommand{\cI}{\mathcal{I}}

\newcommand{\sL}{\mathscr{L}}

\newcommand{\sQ}{\mathscr{Q}}

\newcommand{\sT}{\mathscr{T}}

\newdir{ ir}{{}*!/-5pt/@^{(}} 
\newdir{ il}{{}*!/-5pt/@_{(}} 

\DeclareMathOperator{\supp}{supp}

\newcommand{\dif}{\mathrm d}

\DeclareMathOperator{\rk}{rk}

\newenvironment{sequation}{%
\setcounter{equation}{\value{thm}}%
\numberwithin{equation}{section}%
\begin{equation}%
}{%
\end{equation}%
\numberwithin{equation}{thm}%
\addtocounter{thm}{1}%
}

\numberwithin{equation}{thm}

\DeclareRobustCommand{\SkipTocEntry}[5]{}

\setitemize[1]{leftmargin=*,parsep=0em,itemsep=0.125em,topsep=0.125em}
\setenumerate[1]{leftmargin=*,parsep=0em,itemsep=0.125em,topsep=0.125em}

\newcommand{\iref}[3]{\the\value{#1}.\the\value{#2}(\the\value{#3})}

\definecolor{forrest}{RGB}{81,133,49}
\definecolor{mydarkblue}{RGB}{10,92,153}

\newcommand{\PreprintAndPublication}[2]{%
 \sideremark{%
   \begin{color}{mydarkblue}Preprint\end{color}/%
   \begin{color}{forrest}Publication\end{color}}%
   \begin{color}{mydarkblue}#1\end{color}%
   \begin{color}{forrest}#2\end{color}%
   \sideremark{End of
   \begin{color}{mydarkblue}Preprint\end{color}/%
   \begin{color}{forrest}Publication\end{color}}}

\renewcommand{\PreprintAndPublication}[2]{#1}

\begin{document}

\title[Bogomolov--Sommese vanishing on log canonical pairs]{Bogomolov--Sommese vanishing \\ on log canonical pairs}
\author{Patrick Graf}
\address{Patrick Graf, Mathematisches Institut, Albert-Ludwigs-Universit\"at Frei\-burg, Eckerstrasse 1, 79104 Freiburg im Breisgau, Germany}
\email{\href{mailto:patrick.graf@math.uni-freiburg.de}{patrick.graf@math.uni-freiburg.de}}
\date{\today}
\thanks{The author gratefully acknowledges partial support by the DFG-Forschergruppe 790 ``Classification of Algebraic Surfaces and Compact Complex Manifolds''.}

\begin{abstract}
Let $(X, D)$ be a projective log canonical pair. We show that for any natural number $p$, the sheaf $\bigl(\Omega_X^p(\log \rd D.)\bigr)^{**}$ of reflexive logarithmic $p$-forms does not contain a Weil divisorial subsheaf whose Kodaira--Iitaka dimension exceeds $p$. This generalizes a classical theorem of Bogomolov and Sommese.

In fact, we prove a more general version of this result which also deals with the \emph{orbifoldes g\'eom\'e\-triques} introduced by Campana.
The main ingredients to the proof are the extension theorem of Greb--Kebekus--Kov\'acs--Peternell, a new version of the Negativity lemma, the minimal model program, and a residue map for symmetric differentials on dlt pairs.

We also give an example showing that the statement cannot be generalized to spaces with Du Bois singularities.
As an application, we give a Kodaira--Akizuki--Nakano-type vanishing result for log canonical pairs which holds for reflexive as well as for K\"ahler differentials.
\end{abstract}

\maketitle

\tableofcontents

\section{Introduction and statement of main result}

\PreprintAndPublication{
\subsection{Motivation}

One way to study algebraic varieties is via positivity properties of vector bundles naturally attached to them. The following well-known result in this direction is the starting point of this paper.

\begin{thm}[Bogomolov--Sommese vanishing] \label{thm:BS van}
Let $X$ be a complex projective manifold and $D \subset X$ a divisor with simple normal crossings. For any invertible subsheaf $\ms L \subset \Omega_X^p(\log D)$, we have $\kappa(\ms L) \le p$, where $\kappa(\ms L)$ denotes the Kodaira--Iitaka dimension of $\ms L$. \qed
\end{thm}

The case where $D = 0$ was proved by Bogomolov in his famous paper \cite{Bog79}. The general case is due to Sommese \cite{SS85}, \cite[Cor.~6.9]{EV92}.~--- This result is called a vanishing theorem because it can be rephrased as saying that
\[ H^0\big(X, \sL^{-1} \tensor \Omega_X^p(\log D)\big) = 0 \]
for any complex projective snc pair $(X, D)$ and any invertible sheaf $\sL \in \Pic(X)$ with $\kappa(\sL) > p$.

\subsection*{Importance}
Bogomolov used Theorem \ref{thm:BS van} to prove the inequality $c_1^2 \le 4c_2$ for smooth surfaces of general type \cite[Thm.~5]{Bog79}. This was improved by Miyaoka to $c_1^2 \le 3c_2$, which in turn is optimal \cite[Thm.~4]{Miy77}. Miyaoka's proof also uses Theorem \ref{thm:BS van}.
As another application of Theorem \ref{thm:BS van}, on a smooth projective surface $X$ with $\kappa(X) \ge 0$, Lu and Miyaoka \cite[Thm.~1]{LM95} as well as Bauer et al.~\cite[Thm.~2.6]{BHK+11} gave a lower bound for the self-intersection of a curve $C \subset X$ in terms of its geometric genus and the Chern numbers of $X$.

\subsection*{Generalizations}
First, note that although Bogomolov's original proof relied on taking hyperplane sections, his argument extends to the case of compact K\"ahler manifolds. Also in the K\"ahler setting, Mourougane \cite{Mou98} gave stronger bounds under additional positivity hypotheses on $\sL$.

Likewise, it is natural to consider generalizations of Bogomolov--Sommese vanishing to singular pairs. Note, however, that the statement is false if one allows arbitrary singularities. It turns out that the right setup to consider here is that of a log canonical pair $(X, D)$ together with its sheaves of reflexive differential forms\footnote{A log canonical pair $(X, D)$ consists of a normal variety $X$ and an effective \Q-Weil divisor $D$ on $X$. For the precise definition, see \cite[Sec.~2.3]{KM98}. The sheaf of reflexive differential $p$-forms is then defined to be the double dual of the $p$-th exterior power of the sheaf of logarithmic K\"ahler differentials $\Omega_X(\log \rd D.)$.} $\Omega_X^{[p]}(\log \rd D.)$.
Morally speaking, this is due to the fact that Theorem \ref{thm:BS van} follows from the closedness of global logarithmic forms on a projective snc pair \cite[Thm.~8.35(b)]{Voi02}, and this closedness still holds for reflexive forms on log canonical pairs by the extension theorem of Greb--Kebekus--Kov\'acs--Peternell \cite[Thm.~1.5]{GKKP11}.

Langer was the first to study the singular setting. In \cite[Cor.~4.4]{Lan01}, he proved a version for reduced log canonical surface pairs $(X, D)$, i.e.~pairs where $X$ is a surface and $D$ is a reduced divisor. In \cite[Thm.~4.11]{Lan03}, he claimed nearly everything we prove here. Unfortunately, his proof contains a serious gap which could not be filled -- see \cite[Rem.~1.3]{GKK10}.

Greb, Kebekus, Kov\'acs and Peternell \cite[Thm.~7.2]{GKKP11} were the first to give a generalization which works in arbitrary dimensions. Namely, they proved Bogomolov--Sommese vanishing for $\Q$-Cartier sheaves on log canonical pairs $(X, D)$. Here, a Weil divisorial sheaf $\ms A$ is called $\Q$-Cartier if $\ms A^{[n]}$, the double dual of the $n$-th tensor power, is invertible for some $n > 0$.
Kebekus and Kov\'acs \cite[Cor.~1.3]{KK10b} applied an earlier version of this result \cite[Thm.~8.3]{GKK10} to bounding the variation of families of canonically polarized manifolds in terms of the Kodaira dimension of the base. The interested reader is also referred to the survey \cite{Keb11}.

In the setting of Campana's theory of \emph{orbifoldes g\'eom\'etriques}, Jabbusch and Kebekus \cite[Prop.~7.1]{JK11} proved a version of Bogomolov--Sommese vanishing where the Kodaira dimension $\kappa(\sA)$ is replaced by a slightly larger quantity, the \emph{orbifold Kodaira dimension} $\kappa_\CC(\sA)$. In this theorem, $X$ is even required to be \Q-factorial, i.e.~\emph{any} Weil divisorial sheaf has to be \Q-Cartier. Jabbusch and Kebekus \cite[Thm.~1.5]{JK11} used this to prove a conjecture of Campana in dimension at most three, namely the isotriviality of families of canonically polarized manifolds over bases that are special in the sense of Campana.

The main shortcoming of the theorems \cite[Thm.~7.2]{GKKP11} and \cite[Prop. 7.1]{JK11} is the $\Q$-Cartier assumption, which is not very natural and often hard to work with. For example, in \cite{KK10b} Weil divisorial subsheaves of $\Omega_X^{[p]}(\log \rd D.)$ were constructed as highest exterior powers of destabilizing subsheaves of $\Omega_X^{[1]}(\log \rd D.)$. If $X$ is not $\Q$-factorial, there is no reason why such a sheaf should be $\Q$-Cartier. Since the property of being $\Q$-factorial is not stable under the commonly used operations of taking hyperplane sections, general fibers, and finite covers, this may lead to technical problems.
On the other hand, the assumption that $\sA$ be $\Q$-Cartier is crucial for the proofs of \cite[Thm.~7.2]{GKKP11} and \cite[Prop.~7.1]{JK11}, as they proceed by taking an index one cover with respect to $\sA$.

\subsection{Main result}

We show that the $\Q$-Cartier assumption is not necessary. This settles the question of how much of Bogomolov--Sommese vanishing still holds in the log canonical setting.
The most general version of the result, formulated in the following Theorem \ref{thm:lc BS van}, uses Campana's language of \emph{orbifoldes g\'eom\'etriques}, also called \CC-pairs.
The relevant definitions from Campana's paper \cite{Cam04} are recalled in Section \ref{sec:C-pairs}.

\begin{thm}[Bogomolov--Sommese vanishing on lc \CC-pairs] \label{thm:lc BS van}
Let $(X, D')$ be a complex projective log canonical pair, and let $D \le D'$ be a divisor such that $(X, D)$ is a \CC-pair. If $\sA \subset \Sym_\CC^{[1]} \Omega_X^p(\log D)$ is a Weil divisorial subsheaf, then $\kappa_\CC(\ms A) \le p$.
\end{thm}

The following statement about sheaves of differentials on log canonical pairs is an immediate consequence.

\begin{cor}[Bogomolov--Sommese vanishing on log canonical pairs] \label{cor:lc BS van}
Let $(X, D)$ be a complex projective log canonical pair. If $\ms A \subset \Omega_X^{[p]}(\log \rd D.)$ is a Weil divisorial subsheaf, then $\kappa(\ms A) \le p$. \qed
\end{cor}

In the statement, the Kodaira--Iitaka dimension of a Weil divisorial sheaf appears. This notion is a natural extension of the Kodaira--Iitaka dimension for line bundles. For the reader's convenience, we recall its definition here.

\begin{dfn}[Kodaira--Iitaka dimension, see {\cite[Def.~2.18]{GKKP11}}] \label{dfn:kappa}
Let $X$ be a normal projective variety and $\sA$ a Weil divisorial sheaf on $X$.  If $h^0(X,\, \sA^{[n]}) = 0$ for all $n \in \N$, we say that $\sA$ has \emph{Kodaira--Iitaka dimension} $\kappa(\sA) := -\infty$.  Otherwise, set
\[ M := \bigl\{ n\in \mathbb N \;\big|\; h^0(X,\, \sA^{[n]}) > 0 \bigr\}, \]
note that the restriction of $\sA$ to the smooth locus of $X$ is locally
free by \cite[Lemma 1.1.15]{OSS80}, and consider the natural rational mappings
\[ \phi_n : X \dasharrow \mathbb P\bigl(H^0(X,\, \sA^{[n]})^*\bigr) \quad \text{ for each } n \in M. \]
The Kodaira--Iitaka dimension of $\sA$ is then defined as
\[ \kappa(\sA) := \max_{n \in M} \bigl\{ \dim \overline{\phi_n(X)} \bigr\}. \]
We say that $\sA$ is \emph{big} if $\kappa(\sA) = \dim X$.
\end{dfn}

\subsection{Outline of the proof} \label{subsec:outline of prf}

Since the proof of Theorem \ref{thm:lc BS van} is rather technical, we would like to highlight the main points by considering a simple case first.

\subsubsection{Proof in a simple setting} \label{subsubsec:simple proof}

Let $X$ be a projective cone over an elliptic curve and $D = D' = 0$. Blowing up the vertex gives a log resolution $f\!: Z \to X$, whose exceptional divisor $E$ is isomorphic to the elliptic curve we started with. An elementary calculation shows that $K_X$ is $\Q$-Cartier, and that the discrepancy $a(E, X, 0) = -1$. So $(X, \emptyset)$ is log canonical, but not dlt.
Now let $\ms A \subset \Omega_X^{[1]}$ be an arbitrary Weil divisorial subsheaf. We want to show that $\kappa(\ms A) \le 1$. 

\begin{ntn}[see Section \ref{subsec:refl sheaves}]
Define the \emph{reflexive pullback} $f^{[*]} \ms A$ to be $(f^* \ms A)^{**}$ and the \emph{reflexive tensor powers} $\ms A^{[k]} := (\ms A^{\tensor k})^{**}$, where $^{**}$ denotes the double dual.
\end{ntn}

The main result of \cite{GKKP11} allows us to regard sections in $f^{[*]} \ms A$ as differential forms on $Z$. To be more precise, recall the following theorem.
\begin{thm}[{Extension Theorem, see \cite[Thm.~4.3]{GKKP11}}]
There is a morphism
\[ f^* \Omega_X^{[1]} \to \Omega_Z^1(\log E) \]
which over the smooth locus of $X$ coincides with the usual pullback map of differential forms.
This gives rise to an embedding $f^{[*]} \ms A \inj \Omega_Z^1(\log E)$. \qed
\end{thm}

If $\ms A$ is invertible, then $\kappa(\ms A) = \kappa(f^{[*]} \ms A)$ and we are done by Bogomolov--Sommese vanishing for the snc pair $(Z, E)$. However, it is well-known that taking reflexive powers in general does not commute with reflexive pullback, cf.~\cite{HK04}. There always is an inclusion $(f^{[*]} \ms A)^{[k]} \subset f^{[*]}(\ms A^{[k]})$, but it might be strict. Thus we only have the inequality $\kappa(\ms A) \ge \kappa(f^{[*]} \ms A)$, which is of no use in this situation.\footnote{For a simple example where $\kappa(\ms A) > \kappa(f^{[*]} \ms A)$, take $X$ to be the quadric cone in $\P^3$, $\sA$ the Weil divisorial sheaf associated to a ruling of $X$, and $f$ the blowup of the vertex.}

To remedy this problem, we will enlarge the sheaf $f^{[*]} \ms A$ by taking its \emph{saturation} in $\Omega_Z^1(\log E)$. This is a line bundle $\ms B \subset \Omega_Z^1(\log E)$. For its definition, let $\mc Q$ denote the quotient sheaf $\Omega_Z^1(\log E)/f^{[*]} \ms A$. Then $\ms B$ is defined as the kernel of the composed map
\[ \Omega_Z^1(\log E) \to \mc Q \to \mc Q \big/ \!\Tor \mc Q. \]

If we can show that $\kappa(\ms A) \le \kappa(\ms B)$, then we are done, since $\kappa(\ms B) \le p$. So consider a section $\sigma \in H^0(X, \ms A^{[k]})$. We want to show that $f^* \sigma \in H^0(Z, \ms B^{[k]})$. Since $f^{[*]}(\ms A^{[k]})$ is contained in $\ms B^{[k]}$ away from $E$, we may regard $f^* \sigma \in H^0\big(Z, f^{[*]}(\ms A^{[k]})\big)$ as a rational section of $\ms B^{\tensor k}$, possibly with a pole along $E$. We want to show that in fact there is no such pole. Arguing by contradiction, we assume that $f^* \sigma$ does have a pole, say of order $n > 0$. Then
\[ f^* \sigma \in H^0\big(Z, \ms B^{\tensor k} \tensor \mc O_Z(nE)\big). \]
This induces inclusions of sheaves $\mc O_Z \inj \ms B^{\tensor k}(nE)$ and, twisting and restricting,
\[ \mc O_E(-nE) \inj \ms B^{\tensor k}\big|_E = \big( \ms B\big|_E \big)^{\tensor k}. \]
The restricted map stays injective because $n$ is exactly the pole order of $f^* \sigma$ and no bigger.

The line bundle on the left-hand side has degree $-n \cdot E^2 > 0$, hence is ample, and so is $\ms B\big|_E$. On the other hand, since $\ms B \subset \Omega_Z^1(\log E)$ is saturated by definition, restricting this inclusion to $E$ gives
\[ \ms B\big|_E \subset \Omega_Z^1(\log E)\big|_E, \]
and the vector bundle on the right-hand side sits in the residue sequence
\[ 0 \to \Omega_E^1 \to \Omega_Z^1(\log E)\big|_E \to \mc O_E \to 0. \]
So the ample line bundle $\ms B\big|_E$ injects into either $\Omega_E^1$ or $\mc O_E$. However, since $E$ is an elliptic curve, this is impossible. This is the required contradiction.

\subsubsection{Proof in the general case}

For the general proof, the argument outlined above has to be modified and extended. Several issues arise.
\setitemize[1]{leftmargin=1.7em,parsep=0em,itemsep=0.125em,topsep=0.125em}
\begin{itemize}
\item[---] The philosophy behind the fact that $\mc O_E(-nE)$ is ample is of course that a nonzero effective exceptional divisor $E$ is negative in a certain sense. The usual lemma to this end \cite[Lem.~3.6.2]{BCHM10} says that some component of $E$ is covered by curves intersecting $E$ negatively. This is however not enough for our purposes. We prove the stronger statement that $-E$ is big when restricted to one of the components of $E$, see Proposition \ref{prp:negativity intro}.

\item[---] In general, one cannot assume that $X$ has a log resolution with just a single exceptional divisor, whose discrepancy is exactly $-1$. It turns out that the problem is not the number of exceptional divisors, but the fact that some of them might have discrepancy $> -1$.

Therefore we will pass to a \emph{minimal dlt model} of $X$. This is a ``log crepant'' partial resolution $Z \to X$, i.e.~it only extracts divisors of discrepancy $-1$. The existence of minimal dlt models is due to \cite{BCHM10}. However, although $Z$ is $\Q$-factorial and dlt, things are more complicated than if $Z$ were smooth. For example, the equality $\big( \sB^{[k]}\big|_E \big)^{**} = ( \ms B|_E^{**} )^{[k]}$ does not hold anymore.

We need an analysis of the codimension-two structure of dlt pairs along the reduced boundary, close in spirit to \cite[Sec.~9]{GKKP11} and showing that basically we are dealing here with finite quotient singularities. This allows us to apply Campana's theory of \emph{orbifoldes g\'eo\-m\'e\-triques}, called \emph{\CC-pairs} by Jabbusch and Kebekus. So \CC-pairs appear not only in the statement of the main result, but also as a technical tool in its proof.

\item[---] Generalizing the last step of the simple proof given in \ref{subsubsec:simple proof}, we need to know that for a klt variety $E$ with numerically trivial canonical divisor, $\Omega_E^{[p]}$ does not contain a big subsheaf for any $p$. Since $\dim E = \dim X - 1$, this suggests an approach by induction on the dimension of $X$. However, this is not exactly what we will do. Rather, we will show that we may assume $E$ to be $\Q$-factorial, and then we will apply Bogomolov--Sommese vanishing for $\Q$-factorial varieties \cite[Thm.~7.3]{GKKP11} to $E$.
\end{itemize}
\setitemize[1]{leftmargin=*,parsep=0em,itemsep=0.125em,topsep=0.125em}

\subsection{Further results}

In the course of the proof, we show the following generalization of the Negativity lemma, which we feel might be of independent interest.

\begin{prp}[Negativity lemma for bigness, see Proposition \ref{prp:negativity}] \label{prp:negativity intro}
Let $\pi\!: Y \to X$ be a proper birational morphism between normal quasi-projective varieties. Then for any nonzero effective $\pi$-exceptional $\Q$-Cartier divisor $E$, there is a component $E_0 \subset E$ such that $-E|_{E_0}$ is $\pi|_{E_0}$-big.
\end{prp}

Furthermore, we prove that a residue map exists for symmetric differential forms on dlt pairs. \cite[Thm.~11.7]{GKKP11} essentially is the special case $k = 1$ of this. Note that we do not use \cite[Thm.~11.7]{GKKP11} in the proof.

\begin{thm}[Residues of symmetric differentials, see Theorem \ref{thm:res of symm diff}] \label{thm:res of symm diff intro}
Let $(X, D)$ be a dlt \CC-pair and $D_0 \subset \rd D.$ a component of the reduced boundary. Set $D_0^c := \Diff_{D_0}(D - D_0)$. Then the pair $(D_0, D_0^c)$ is also a dlt \CC-pair, and for any integer $p \ge 1$, there is a map
\[ \res_{D_0}^k\!: \Sym_\CC^{[k]} \Omega_X^p(\log D) \to \Sym_\CC^{[k]} \Omega_{D_0}^{p-1}(\log D_0^c) \]
which on the snc locus of $(X, \rp D.)$ coincides with the $k$-th symmetric power of the usual residue map for snc pairs.
\end{thm}
In Theorem \ref{thm:res of symm diff intro}, the divisor $\Diff_{D_0}(D - D_0)$ is the \emph{different divisor}, whose definition is recalled in Proposition/Definition \ref{prpdfn:different}.
The symbol $\Sym_\CC^{[k]} \Omega_X^p(\log D)$ denotes the sheaf of orbifold pluri-differentials. Its definition is recalled in Definition \ref{dfn:C-diff}.

\subsection{Applications of Theorem \ref{thm:lc BS van}}

We prove the following corollary, which is related to a conjecture of Campana -- see Section \ref{sec:rem on kappa++}.

\begin{cor}[see Corollary \ref{cor:easy lc kappa++}] \label{cor:easy lc kappa++ intro}
Let $(X, D)$ be a complex projective log canonical pair. Set
\begin{align*}
\kappa_{++}(X, D) := \max \bigl\{ \kappa(\sA) \;\big|\; \sA \text{ a Weil divisorial subsheaf of $\Omega_X^{[p]}(\log \rd D.)$, \phantom{.}} & \\
\text{for some $p > 0$} \bigr\}. &
\end{align*}
If $\kappa_{++}(X, D) = \dim X$, then $(X, D)$ is of log general type.
\end{cor}

Using Serre duality, we prove a Kodai\-ra--Akizuki--Nakano-type vanishing result for the top cohomology groups of certain sheaves of twisted reflexive differentials. Due to dimension reasons, this also holds for K\"ahler differentials.

\begin{cor}[KAN-type vanishing, see Corollaries \ref{cor:Serre dual lc BS van} and \ref{cor:Serre dual lc BS van II}] \label{cor:Serre dual lc BS van intro}
Let $(X, D)$ be a complex projective log canonical pair of dimension $n$ and $\sA$ a Weil divisorial sheaf on $X$. Then
\[ \begin{array}{llll}
H^n \bigl( X, \bigl( \Omega_X^{[p]}(\log \rd D.) \tensor \sA \bigr)^{**} \bigr) & \!\! = & \!\! 0 & \text{and} \\[1ex]
H^n \bigl( X, \Omega_X^p(\log \rd D.) \tensor \sA \bigr)                        & \!\! = & \!\! 0 & \\
\end{array} \]
for all $p > n - \kappa(\sA)$.
\end{cor}
}{
\subsection{Motivation}

One way to study algebraic varieties is via positivity properties of vector bundles naturally attached to them. The following well-known result in this direction is the starting point of this paper.

\begin{thm}[Bogomolov--Sommese vanishing] \label{thm:BS van}
Let $X$ be a complex projective manifold and $D \subset X$ a divisor with simple normal crossings. For any invertible subsheaf $\sL \subset \Omega_X^p(\log D)$, we have $\kappa(\sL) \le p$, where $\kappa(\sL)$ denotes the Kodaira--Iitaka dimension of $\sL$. \qed
\end{thm}

The case where $D = 0$ was proved by Bogomolov in his famous paper \cite{Bog79}. The general case is due to Sommese \cite[Cor.~6.9]{EV92}.
Bogomolov used Theorem \ref{thm:BS van} to prove the inequality $c_1^2 \le 4c_2$ for smooth surfaces of general type \cite[Thm.~5]{Bog79}. This was improved by Miyaoka to $c_1^2 \le 3c_2$, which in turn is optimal \cite[Thm.~4]{Miy77}. Miyaoka's proof also uses Theorem \ref{thm:BS van}.

\subsection*{Generalizations}
It is natural to consider generalizations of Bogomolov--Sommese vanishing to singular pairs. Note, however, that the statement is false if one allows arbitrary singularities. It turns out that the right setup to consider here is that of a log canonical pair $(X, D)$ together with its sheaves of reflexive differential forms\footnote{A log canonical pair $(X, D)$ consists of a normal variety $X$ and an effective \Q-Weil divisor $D$ on $X$. For the precise definition, see \cite[Sec.~2.3]{KM98}. The sheaf of reflexive differential $p$-forms is then defined to be the double dual of the $p$-th exterior power of the sheaf of logarithmic K\"ahler differentials $\Omega_X(\log \rd D.)$.} $\Omega_X^{[p]}(\log \rd D.)$.
Morally speaking, this is due to the fact that Theorem \ref{thm:BS van} follows from the closedness of global logarithmic forms on a projective snc pair \cite[Thm.~8.35(b)]{Voi02}, and this closedness still holds for reflexive forms on log canonical pairs by the extension theorem of Greb--Kebekus--Kov\'acs--Peternell \cite[Thm.~1.5]{GKKP11}.

Greb, Kebekus, Kov\'acs and Peternell \cite[Thm.~7.2]{GKKP11} were the first to give a generalization which works in arbitrary dimensions. Namely, they proved Bogomolov--Sommese vanishing for $\Q$-Cartier sheaves on log canonical pairs $(X, D)$. Here, a Weil divisorial sheaf $\ms A$ is called $\Q$-Cartier if $\ms A^{[n]}$, the double dual of the $n$-th tensor power, is invertible for some $n > 0$.
Kebekus and Kov\'acs \cite[Cor.~1.3]{KK10b} applied an earlier version of this result \cite[Thm.~8.3]{GKK10} to bounding the variation of families of canonically polarized manifolds in terms of the Kodaira dimension of the base. The interested reader is also referred to the survey \cite{Keb11}.

The main shortcoming of the theorem \cite[Thm.~7.2]{GKKP11} is the \Q-Cartier assumption, which is not very natural and often hard to work with. For example, in \cite{KK10b} Weil divisorial subsheaves of $\Omega_X^{[p]}(\log \rd D.)$ were constructed as highest exterior powers of destabilizing subsheaves of $\Omega_X^{[1]}(\log \rd D.)$. If $X$ is not $\Q$-factorial, there is no reason why such a sheaf should be $\Q$-Cartier. Since the property of being $\Q$-factorial is not stable under the commonly used operations of taking hyperplane sections, general fibers, and finite covers, this may lead to technical problems.
On the other hand, the assumption that $\sA$ be $\Q$-Cartier is crucial for the proof of \cite[Thm.~7.2]{GKKP11}, as it proceeds by taking an index one cover with respect to~$\sA$.

\subsection{Main result}

We show that the $\Q$-Cartier assumption is not necessary. This settles the question of how much of Bogomolov--Sommese vanishing still holds in the log canonical setting.
The most general version of the result, formulated in the following Theorem \ref{thm:lc BS van}, uses Campana's language of \emph{orbifoldes g\'eom\'etriques}, also called \CC-pairs.
The relevant definitions from Campana's paper \cite{Cam04} are recalled in Section \ref{sec:C-pairs}.

\begin{thm}[Bogomolov--Sommese vanishing on lc \CC-pairs] \label{thm:lc BS van}
Let $(X, D')$ be a complex projective log canonical pair, and let $D \le D'$ be a divisor such that $(X, D)$ is a \CC-pair. If $\sA \subset \Sym_\CC^{[1]} \Omega_X^p(\log D)$ is a Weil divisorial subsheaf, then $\kappa_\CC(\ms A) \le p$.
\end{thm}

The following statement about sheaves of differentials on log canonical pairs is an immediate consequence.

\begin{cor}[Bogomolov--Sommese vanishing on log canonical pairs] \label{cor:lc BS van}
Let $(X, D)$ be a complex projective log canonical pair. If $\ms A \subset \Omega_X^{[p]}(\log \rd D.)$ is a Weil divisorial subsheaf, then $\kappa(\ms A) \le p$. \qed
\end{cor}

In the statement, the Kodaira--Iitaka dimension of a Weil divisorial sheaf appears. This notion is a natural extension of the Kodaira--Iitaka dimension for line bundles. For the reader's convenience, we recall its definition here.

\begin{dfn}[Kodaira--Iitaka dimension, see {\cite[Def.~2.18]{GKKP11}}] \label{dfn:kappa}
Let $X$ be a normal projective variety and $\sA$ a Weil divisorial sheaf on $X$.  If $h^0(X,\, \sA^{[n]}) = 0$ for all $n \in \N$, we say that $\sA$ has \emph{Kodaira--Iitaka dimension} $\kappa(\sA) := -\infty$.  Otherwise, set
\[ M := \bigl\{ n\in \mathbb N \;\big|\; h^0(X,\, \sA^{[n]}) > 0 \bigr\}, \]
note that the restriction of $\sA$ to the smooth locus of $X$ is locally
free by \cite[Lemma 1.1.15]{OSS80}, and consider the natural rational mappings
\[ \phi_n : X \dasharrow \mathbb P\bigl(H^0(X,\, \sA^{[n]})^*\bigr) \quad \text{ for each } n \in M. \]
The Kodaira--Iitaka dimension of $\sA$ is then defined as
\[ \kappa(\sA) := \max_{n \in M} \bigl\{ \dim \overline{\phi_n(X)} \bigr\}. \]
We say that $\sA$ is \emph{big} if $\kappa(\sA) = \dim X$.
\end{dfn}

\subsection{Outline of the proof} \label{subsec:outline of prf}

The main idea for the proof of Theorem \ref{thm:BS van} is best explained in a simple case. Let $X$ be a projective cone over an elliptic curve and $D = D' = 0$. Blowing up the vertex gives a log resolution $f\!: Z \to X$, whose exceptional divisor $E$ is isomorphic to the elliptic curve we started with. An elementary calculation shows that $K_X$ is $\Q$-Cartier, and that the discrepancy $a(E, X, 0) = -1$. So $(X, \emptyset)$ is log canonical, but not dlt.
Now let $\ms A \subset \Omega_X^{[1]}$ be an arbitrary Weil divisorial subsheaf. We want to show that $\kappa(\ms A) \le 1$. 

The main result of \cite{GKKP11} allows us to regard sections in $f^{[*]} \sA := (f^* \sA)^{**}$ as differential forms on $Z$. To be more precise, recall the following theorem.
\begin{thm}[{Extension Theorem, see \cite[Thm.~4.3]{GKKP11}}]
There is a morphism
\[ f^* \Omega_X^{[1]} \to \Omega_Z^1(\log E) \]
which over the smooth locus of $X$ coincides with the usual pullback map of differential forms.
This gives rise to an embedding $f^{[*]} \ms A \inj \Omega_Z^1(\log E)$. \qed
\end{thm}

Hence $\kappa(f^{[*]} \sA) \le 1$ by Theorem \ref{thm:BS van}. However, since reflexive pullback does not commute with reflexive tensor powers, in this situation in general we only have the inequality $\kappa(f^{[*]} \sA) \le \kappa(\sA)$.\footnote{For a simple example where $\kappa(\ms A) > \kappa(f^{[*]} \ms A)$, take $X$ to be the quadric cone in $\P^3$, $\sA$ the Weil divisorial sheaf associated to a ruling of $X$, and $f$ the blowup of the vertex.}
Therefore we enlarge the sheaf $f^{[*]} \sA$ by taking its saturation $\sB \subset \Omega_Z^1(\log E)$. We prove that sections of $\sA^{[k]} := (\sA^{\otimes k})^{**}$ extend to sections of $\sB^{[k]}$. Then $\kappa(\sA) = \kappa(\sB)$, and we are done.

The proof is by contradiction: Assuming that some section of $\sA^{[k]}$ acquires a pole when being pulled back, we use the fact that $E \subset Z$ has negative self-intersection to deduce that $\sB|_E$ is ample. On the other hand, the residue sequence for the pair $(Z, E)$ shows that $\sB|_E$ injects into the trivial line bundle. This yields the desired contradiction.

For the general proof, the fact that contractible curves have negative self-inter\-section is replaced by a generalized Negativity lemma, see Proposition \ref{prp:negativity intro} below. A more serious issue is that we cannot really work on a log resolution, because it ``extracts too many divisors''. Instead, we have to pass to a \emph{minimal dlt model} $(Z, D_Z) \to (X, D)$, a partial resolution of $(X, D)$ which extracts only divisors of discrepancy exactly $-1$. However, $(Z, D_Z)$ is not an snc pair, but only a dlt pair, which makes the proof technically rather involved. In particular, we need to know that a residue sequence still exists for dlt pairs. This is provided by Theorem~\ref{thm:res of symm diff intro} below.

\subsection{Further results}

In the course of the proof, we show the following generalization of the Negativity lemma, which we feel might be of independent interest.

\begin{prp}[Negativity lemma for bigness, see Proposition \ref{prp:negativity}] \label{prp:negativity intro}
Let $\pi\!: Y \to X$ be a proper birational morphism between normal quasi-projective varieties. Then for any nonzero effective $\pi$-exceptional $\Q$-Cartier divisor $E$, there is a component $E_0 \subset E$ such that $-E|_{E_0}$ is $\pi|_{E_0}$-big.
\end{prp}

Furthermore, we prove that a residue map exists for symmetric differential forms on dlt pairs. \cite[Thm.~11.7]{GKKP11} essentially is the special case $k = 1$ of this. Note that we do not use \cite[Thm.~11.7]{GKKP11} in the proof.

\begin{thm}[Residues of symmetric differentials, see Theorem \ref{thm:res of symm diff}] \label{thm:res of symm diff intro}
Let $(X, D)$ be a dlt \CC-pair and $D_0 \subset \rd D.$ a component of the reduced boundary. Set $D_0^c := \Diff_{D_0}(D - D_0)$. Then the pair $(D_0, D_0^c)$ is also a dlt \CC-pair, and for any integer $p \ge 1$, there is a map
\[ \res_{D_0}^k\!: \Sym_\CC^{[k]} \Omega_X^p(\log D) \to \Sym_\CC^{[k]} \Omega_{D_0}^{p-1}(\log D_0^c) \]
which on the snc locus of $(X, \rp D.)$ coincides with the $k$-th symmetric power of the usual residue map for snc pairs.
\end{thm}
In Theorem \ref{thm:res of symm diff intro}, the divisor $\Diff_{D_0}(D - D_0)$ is the \emph{different divisor}, whose definition is recalled in Proposition/Definition \ref{prpdfn:different}.
The symbol $\Sym_\CC^{[k]} \Omega_X^p(\log D)$ denotes the sheaf of orbifold pluri-differentials. Its definition is recalled in Definition \ref{dfn:C-diff}.

\subsection{An application of Theorem \ref{thm:lc BS van}}

Using Serre duality, we prove a Kodai\-ra--Akizuki--Nakano-type vanishing result for the top cohomology groups of certain sheaves of twisted reflexive differentials. Due to dimension reasons, this also holds for K\"ahler differentials.

\begin{cor}[KAN-type vanishing, see Corollaries \ref{cor:Serre dual lc BS van} and \ref{cor:Serre dual lc BS van II}] \label{cor:Serre dual lc BS van intro}
Let $(X, D)$ be a complex projective log canonical pair of dimension $n$ and $\sA$ a Weil divisorial sheaf on $X$. Then
\[ \begin{array}{llll}
H^n \bigl( X, \bigl( \Omega_X^{[p]}(\log \rd D.) \tensor \sA \bigr)^{**} \bigr) & \!\! = & \!\! 0 & \text{and} \\[1ex]
H^n \bigl( X, \Omega_X^p(\log \rd D.) \tensor \sA \bigr)                        & \!\! = & \!\! 0 & \\
\end{array} \]
for all $p > n - \kappa(\sA)$.
\end{cor}
}

\subsection{Outline of the paper}

This paper is divided into four parts.

In Part \ref{part:I}, we fix notation and recall some technical results concerning reflexive sheaves. Furthermore we prove the Negativity lemma for bigness, as mentioned above in Proposition \ref{prp:negativity intro}.

Part \ref{part:II} is devoted to a discussion of dlt pairs and their \CC-differentials. In particular, we prove Theorem \ref{thm:res of symm diff intro} and a weak version of Bogomolov--Sommese vanishing on dlt \CC-pairs, which is applied in the main proof.

Part \ref{part:III} contains the proof of the Main Theorem \ref{thm:lc BS van}, as well as an example showing that it fails as soon as one relaxes the assumption on log canonicity to $X$ having Du Bois singularities.

\PreprintAndPublication{
Finally, in Part \ref{part:IV} we prove Corollaries \ref{cor:easy lc kappa++ intro} and \ref{cor:Serre dual lc BS van intro}.
}{
Finally, in Part \ref{part:IV} we prove Corollary \ref{cor:Serre dual lc BS van intro}.
}

\subsection{Acknowledgments}

The results of this paper are part of the author's Ph.D. thesis. He would like to thank his supervisor, Stefan Kebekus, for fruitful advice and support. He would also like to thank the members of Kebekus' research group, especially Daniel Greb, Clemens J\"order, and Alex K\"uronya, as well as the research group's past guests, especially Kiwamu Watanabe, for interesting discussions.

\PreprintAndPublication{
The author is particularly grateful to Stefan Kebekus, Alex K\"uronya, and the anonymous referee for careful proofreading and valuable suggestions.
}{
The author is particularly grateful to Stefan Kebekus, Alex K\"uronya, and the anonymous referee for valuable suggestions concerning the exposition of this paper.
}

\part{TECHNICAL PREPARATIONS} \label{part:I}

\section{Setup of notation} \label{sec:notation}

\subsection{Base field} \label{subsec:base field}

\PreprintAndPublication{
Throughout this paper, we work over the field of complex numbers $\C$. Even the smooth version of Bogomolov--Sommese vanishing, Theorem \ref{thm:BS van}, is known to be false over fields of positive characteristic -- see \cite[Rem.~7.1]{Kol95} or the remark at the end of \cite[\S 12]{Bog79}.
}{
Throughout this paper, we work over the field of complex numbers $\C$. Even the smooth version of Bogomolov--Sommese vanishing, Theorem \ref{thm:BS van}, is known to be false over fields of positive characteristic -- see the remark at the end of \cite[\S 12]{Bog79}.
}

\subsection{Reflexive sheaves} \label{subsec:refl sheaves}
Let $X$ be a normal variety. For any coherent sheaf $\ms E$ on $X$, $\ms E^{**}$ denotes the \emph{double dual} of $\ms E$. The sheaf $\ms E$ is called \emph{reflexive} if the canonical map $\ms E \to \ms E^{**}$ is an isomorphism. We use square brackets $^{[\cdot]}$ as a shorthand notation for the double dual, i.e.~$\ms E^{[k]}$ means the \emph{reflexive tensor power} $(\ms E^{\tensor k})^{**}$, and analogously for $\bigwedge^{[k]} \ms E$ and $\Sym^{[k]} \ms E$. A \emph{Weil divisorial sheaf} is a reflexive sheaf of rank 1. A Weil divisorial sheaf $\ms A$ is called \emph{$\Q$-Cartier} if $\ms A^{[k]}$ is invertible for some $k > 0$.

If $f\!: Y \to X$ is any morphism from a normal variety $Y$, then the \emph{reflexive pullback} $f^{[*]} \ms E := (f^* \ms E)^{**}$. If $f$ is a closed immersion, we write $\ms E\big|_Y^{**}$ instead.

Let $D$ be an effective reduced divisor on $X$. Then the sheaf of \emph{reflexive differential $p$-forms} is $\Omega_X^{[p]}(\log D) := \left(\bigwedge^p \Omega_X^1(\log D)\right)^{**}$, where $\Omega_X^1(\log D)$ is the sheaf of logarithmic K\"ahler differentials.

A subsheaf $\ms A$ of a reflexive sheaf $\ms E$ is said to be \emph{saturated} if the quotient $\ms E/\ms A$ is torsion-free.

\PreprintAndPublication{
As a general reference for this material, we recommend \cite{OSS80}, \cite{Sch10} and \cite{GKKP11}.
}{
As a general reference for this material, we recommend \cite{OSS80} and \cite{GKKP11}.
}

\subsection{Birational Geometry} \label{subsec:biratl geom}

Nothing else being said, by a divisor we mean a Weil divisor with $\R$-coefficients. For $k \in \{ \Q, \R \}$, a divisor is called \emph{$k$-Cartier} if it is a $k$-linear combination of integral Cartier divisors. A \emph{pair} $(X, D)$ consists of a normal variety $X$ and an effective divisor $D$ on $X$. The pair is called \emph{reduced} if $D$ is reduced, i.e.~if all nonzero coefficients of $D$ are equal to $1$. For the definitions of \emph{log canonical} and \emph{dlt} pairs, see \cite[Sec.~2.3]{KM98}. Note that the definitions given there only apply if $D$ is a $\Q$-divisor, but they carry over without change to the case where $D$ is allowed to have arbitrary real coefficients.

Let $f\!: Y \to X$ be a proper birational morphism between normal varieties. The \emph{exceptional locus} of $f$ is denoted by $\Exc(f)$. Let $D$ be a divisor on $X$. Then $f^{-1}_{\;\;\;*} D$ is the \emph{strict transform} of $D$ on $Y$. The \emph{round-down} of $D$, defined by rounding down the coefficients, is denoted $\rd D.$. Ditto for the \emph{round-up} $\rp D.$. The \emph{fractional part}~$\{ D \}$ is, by definition, $D - \rd D.$.

Now let $f\!: Y \to X$ be an arbitrary proper morphism between $Y$ and $X$. Two divisors $D_1$ and $D_2$ on $Y$ are said to be \emph{$\R$-linearly equivalent over $X$}, written $D_1 \sim_{\R, X} D_2$, if their difference is an $\R$-linear combination of principal divisors, and Cartier divisors pulled back by $f$.

The smooth and singular loci of $X$ are denoted $X_\sm$ and $X_\sg$, respectively. The non-snc locus of a pair $(X, D)$ is denoted $(X, D)_\nsnc$. A subset $Z \subset X$ is called \emph{small} if it is closed and $\codim_X Z \ge 2$. An open subset $U \subset X$ is called \emph{big} if its complement is small.

\section{Reflexive sheaves} \label{sec:refl sheaves}

In this section, we gather some technical results concerning reflexive sheaves and their saturated subsheaves. All of this is well-known. The reader familiar with this topic can go directly to Section \ref{sec:neg lemma}. Throughout this section, $X$ denotes a normal variety, and all sheaves occurring are tacitly assumed to be quasi-coherent.

\begin{lem}[Singular loci of torsion-free sheaves] \label{lem:sg locus of tor-free}
Let $\ms F$ be a torsion-free sheaf on $X$. Then there is a big open subset $U \subset X$ such that $\ms F\big|_U$ is locally free.
\end{lem}

\begin{proof}
Removing the singular locus of $X$, which is a small subset, we may assume that $X$ is smooth. Then the assertion follows from \cite[Cor.~on p.~148]{OSS80}.
\end{proof}

\begin{dfn}[Normal sheaves]
A sheaf $\ms F$ on $X$ is called \emph{normal} if for any open set $U \subset X$ and any small subset $Z \subset U$, the restriction map
\[ \mathit\Gamma(U, \sF) \to \mathit\Gamma(U \minus Z, \sF) \]
is an isomorphism.
\end{dfn}

\begin{rem}
The connection between normal sheaves and normal varieties is as follows: a variety $Y$ is normal if and only if it is smooth in codimension one and $\O_Y$ is normal \cite[Thm.~in (3.18) on p.~368]{Rei87}.
\end{rem}

\begin{lem}[Characterization of reflexive sheaves] \label{lem:char reflexive}
A sheaf $\ms F$ on $X$ is reflexive if and only if it is torsion-free and normal.
\end{lem}

\begin{proof}
For smooth $X$, this is \cite[Lem.~1.1.12]{OSS80}. If $X$ is only normal, the proof still applies verbatim.
\end{proof}

\begin{lem}[Reflexivity of saturation] \label{lem:saturation reflexive}
Let $\ms A \subset \ms E$ be sheaves on $X$, with $\ms E$ reflexive and $\ms A$ saturated in $\ms E$. Then $\ms A$ is reflexive, too.
\end{lem}

\begin{proof}
By Lemma \ref{lem:char reflexive}, $\ms E$ is torsion-free. Hence so is $\ms A$, since it is a subsheaf of $\ms E$. In addition, by \cite[Lem.~1.1.16]{OSS80} $\ms A$ is normal\footnote{Although \cite{OSS80} only deals with smooth $X$, the argument in the general case is the same.}. Now apply Lemma \ref{lem:char reflexive} to $\ms A$.
\end{proof}

\begin{lem}[Subbundle property] \label{lem:saturation-injective}
Let $\ms A \subset \ms E$ be reflexive sheaves on $X$, with $\ms A$ saturated in $\ms E$. Then there is a big open subset $U \subset X$ such that both $\ms A\big|_U$ and $\ms E\big|_U$ are locally free, and $\ms A\big|_U$ is a sub-vector bundle of $\ms E\big|_U$.
\end{lem}

\begin{proof}
By Lemma \ref{lem:sg locus of tor-free}, there is a big open set $U \subset X$ such that the sheaves $\ms A\big|_U$, $\ms E\big|_U$ and $(\ms E/\ms A)\big|_U$ are locally free. Then the short exact sequence
\[ 0 \to \ms A\big|_U \to \ms E\big|_U \to (\ms E/\ms A)\big|_U \to 0 \]
is locally split. The claim follows.
\end{proof}

\begin{dfn}[Rational sections] \label{dfn:ratl section}
A \emph{rational section} $s$ of a reflexive sheaf $\ms F$ on $X$ is an element of $\mathit\Gamma(U, \ms F)$ for some dense open subset $U \subset X$. The rational section $s$ is called regular if it lies in the image of the restriction map $\mathit\Gamma(X, \sF) \to \mathit\Gamma(U, \sF)$.
\end{dfn}

\begin{dfn}[Sheaf of sections with arbitrary poles] \label{dfn:383}
If $X$ is a normal variety, $D \subset X$ a reduced Weil divisor, and $\sF$ a reflexive sheaf on $X$, we denote by $\sF(*D) := \varinjlim \bigl( \sF \tensor \O_X(mD) \bigr)^{**}$ the (quasi-coherent) sheaf of sections of $\sF$ with arbitrary poles along $D$.
\end{dfn}

\begin{rem}
If $X$ is a normal variety, $\sF$ a reflexive sheaf on $X$, and $U \subset X$ a big open set, then the restriction map $\mathit\Gamma(X, \sF) \to \mathit\Gamma(U, \sF)$ is bijective.
If $U$ is not big and $D$ denotes the codimension-one part of $X \minus U$, then any section $s \in \mathit\Gamma(U, \sF)$ extends uniquely to a section $\bar s \in \mathit\Gamma(X, \sF(*D))$.
\end{rem}

\begin{lem}[Invariance of pole orders] \label{lem:saturation-poles}
Let $\ms E$ be a reflexive sheaf on $X$, with a saturated subsheaf $\ms A \subset \ms E$. Let $s$ be a rational section of $\ms A$. If $s$ is regular as a section of $\ms E$, then $s$ is also regular as a section of $\sA$.
\end{lem}

\begin{proof}
Let $s$ be the restriction of some $\bar s \in H^0(X, \sE)$, and let $s'$ be the image of $\bar s$ in $H^0(X, \sE/\sA)$. By Definition \ref{dfn:ratl section}, $s \in \mathit\Gamma(U, \sA)$ for some dense open $U \subset X$. We then have $s'|_U = 0$. It follows that $s' = 0$, because $\sE/\sA$ is torsion-free. Hence $\bar s \in H^0(X, \ms A)$.
\end{proof}

\begin{lem}[Generically equal subsheaves] \label{lem:gen eq}
Let $\ms E$ be a reflexive sheaf on $X$, with two subsheaves $\sA, \sB \subset \sE$, where $\sB$ is saturated in $\sE$. Suppose that for some dense open subset $U \subset X$, we have that $\sA|_U$ and $\sB|_U$ are equal as subsheaves of $\sE|_U$. Then $\sA \subset \sB$ as subsheaves of $\sE$.
\end{lem}

\begin{proof}
Let $s$ be a section of $\sA$. Then $s$ is a rational section of $\sB$, which is regular as a section of $\sE$. By Lemma \ref{lem:saturation-poles}, we conclude that $s$ is a regular section of $\sB$.
\end{proof}

\begin{lem}[Restriction to a divisor] \label{lem:saturation-restriction}
Let $\ms A \subset \ms E$ be reflexive sheaves on $X$. If $\ms A$ is saturated in $\ms E$ and $D \subset X$ is a prime divisor such that $\supp D$ is normal, then the induced map $\alpha\!: \ms A\big|_D^{**} \to \ms E\big|_D^{**}$ is injective.
\end{lem}

\begin{proof}
By Lemma \ref{lem:saturation-injective}, the map $\alpha$ is generically injective. But then it is already injective, since $\ms A\big|_D^{**}$ is torsion-free by Lemma \ref{lem:char reflexive}.
\end{proof}

\begin{lem}[Extending morphisms] \label{lem:ext morphisms}
Let $\ms A, \ms B$ be reflexive sheaves on $X$. Then:
\begin{enumerate}
\item\label{itm:extmor1} Any sheaf morphism $\alpha\!: \ms A\big|_U \to \ms B\big|_U$ on some big open subset $U \subset X$ extends uniquely to a sheaf morphism $\overline\alpha\!: \ms A \to \ms B$.
\item\label{itm:extmor2} If $\alpha$ is an isomorphism, then so is $\overline\alpha$. Ditto for $\alpha$ injective.
\end{enumerate}
\end{lem}

\begin{proof}
Let $V \subset X$ be any open set. Denoting by $\rho_{\cdot,\cdot}$ the restriction maps, we must have
\[ \rho_{V, V \cap U}^\sB \circ \overline\alpha(V) = \alpha(V \cap U) \circ \rho_{V, V \cap U}^\sA. \]
By Lemma \ref{lem:char reflexive}, $\rho_{V, V \cap U}^\sB$ is an isomorphism, hence there is a unique way of defining $\overline\alpha(V)$.

As for the second part, if $\alpha$ is an isomorphism, then the inverse of $\overline\alpha$ is given by $\overline{\alpha^{-1}}$. And if $\alpha$ is injective, then $\overline\alpha$ is generically injective, hence injective by the torsion-freeness of $\sA$.
\end{proof}

\begin{lem}[Reduction lemma] \label{lem:reduction}
Let $\ms A, \ms E, \ms F, \ms G$ be reflexive sheaves on $X$, with $\ms A$ of rank 1. Assume that on some big open set $U \subset X$ where $\ms E, \ms F$ and $\ms G$ are locally free, there is a short exact sequence
\begin{equation} \label{eqn:ses|_U}
0 \to \ms F\big|_U \to \ms E\big|_U \to \ms G\big|_U \to 0.
\end{equation}
If there exists an inclusion $\ms A \inj \bigwedge^{[r]} \ms E$, then
\[ \textstyle \ms A \inj \big(\bigwedge^{[j]} \ms F \tensor \bigwedge^{[r-j]} \ms G\big)^{**} \]
for some $0 \le j \le r$.
\end{lem}

\begin{proof}
Applying \cite[Ch.~II, Ex.~5.16(d)]{Har77} to the sequence (\ref{eqn:ses|_U}), we see that $\bigwedge^r \ms E\big|_U$ has a filtration with quotients
\[ \phantom{0 \le j \le r \quad}
\textstyle \big( \bigwedge^j \ms F \tensor \bigwedge^{r-j} \ms G \big) \big|_U, \quad 0 \le j \le r. \]
Clearly there is a nonzero map $\ms A\big|_U \to \big( \bigwedge^j \ms F \tensor \bigwedge^{r-j} \ms G \big) \big|_U$ for some $j$. Since the sheaf on the right-hand side is locally free, this map is even generically injective. Extending it to all of $X$ by Lemma \ref{lem:ext morphisms}(\ref{itm:extmor1}), we get the desired map
\[ \textstyle \ms A \to \big(\bigwedge^{[j]} \ms F \tensor \bigwedge^{[r-j]} \ms G\big)^{**}. \]
It is injective because it is generically so and $\ms A$ is torsion-free.
\end{proof}

\begin{rem}[Passing to a power of a section] \label{rem:passing to a power} Let
  $f\!: Z \to X$ be a proper birational morphism of normal varieties, with
  exceptional locus $E$. Let $\sA$ be a Weil divisorial sheaf on $X$ and $\sB$ a
  Weil divisorial sheaf on $Z$ such that $\sB|_{Z \minus E} \isom \sA|_{X \minus
    f(E)}$. Let $\sigma \in H^0(X, \sA^{[k]})$ be a section, and choose a
  natural number $\ell$. 

  Since $\sB|_{Z \minus E} \isom \sA|_{X \minus f(E)}$, we can regard $f^*
  \sigma \in H^0 \bigl( Z, f^{[*]} (\sA^{[k]}) \bigr)$ as a rational section of
  $\sB^{[k]}$. Raising it to the $\ell$-th power, $(f^* \sigma)^\ell$ is a
  rational section of $(\sB^{[k]})^{[\ell]} = \sB^{[k \cdot \ell]}$. Similarly,
  pulling back $\sigma^\ell \in H^0(X, \sA^{[k \cdot \ell]})$ gives a rational
  section $f^* (\sigma^\ell)$ of $\sB^{[k \cdot \ell]}$. We claim that
  \begin{sequation}
    f^* (\sigma^\ell) = (f^* \sigma)^\ell
  \end{sequation}%
  as rational sections of $\sB^{[k \cdot \ell]}$. But this is clear since they
  agree on $Z \minus E$.
\end{rem}

\section{A Negativity lemma for bigness} \label{sec:neg lemma}

\subsection{Statement of result}

In this section, we prove a lemma about the negativity of exceptional divisors which extends the well-known statement of \cite[Lem.~3.6.2(1)]{BCHM10} and \cite[Lem.~3.39]{KM98}. Since our lemma deals with bigness instead of nefness, it is much better suited for discussing big subsheaves of the cotangent bundle. It also allows to give alternative proofs of the aforementioned lemmas which do not rely on deep results like Serre duality and Riemann--Roch for surfaces.

\begin{prp}[Negativity lemma for bigness] \label{prp:negativity}
Let $\pi\!: Y \to X$ be a proper birational morphism between normal quasi-projective varieties such that the $\pi$-exceptional set is of (not necessarily pure) codimension one. Then for any nonzero effective $\pi$-exceptional $\Q$-Cartier divisor $E$, there is a component $E_0 \subset E$ such that $-E\big|_{E_0}$ is $\pi\big|_{E_0}$-big.
\end{prp}

\begin{rem}
Recall that in the relative setting, a divisor is called $f$-big if its restriction to a \emph{general} fiber of $f$ is big. See \cite[Def.~3.1.1(7)]{BCHM10}. In particular, if $f$ is birational then any divisor is $f$-big.
\end{rem}

Before we can give the proof of Proposition \ref{prp:negativity}, we need a little preparation.

\begin{lem} \label{lem:Weil le Cartier}
Let $X$ be a normal quasi-projective variety and $D$ an effective integral Weil divisor on $X$. Then there is a Cartier divisor $D'$ on $X$ such that $D' \ge D$, and such that $\supp(D)$ and $\supp(D'-D)$ have no components in common.
\end{lem}

\begin{proof}
We may assume that $X \subset \P^n$ is projective. Write $D = \sum a_i D_i$, where the $D_i$ are distinct prime divisors. Let $\pi$ be a finite surjective map from $X$ to a smooth space such that $D$ is not contained in the ramification locus of $\pi$ and $\pi(D_i) \ne \pi(D_j)$ for $i \ne j$. Here $\pi(D_i)$ denotes the reduced divisor associated with $\supp \pi_* D_i$.
For example, one may take $\pi$ to be a projection to a general linear subspace $L \subset \P^n$, where $\dim L = \dim X$. Now, choosing $D' = \sum a_i \pi^* \pi(D_i)$ proves the lemma.
\end{proof}

\subsection{Proof of Proposition \ref{prp:negativity}}

Let a map $\pi$ and a divisor $E$ be given as in the statement of Proposition \ref{prp:negativity}. Let $H \subset Y$ be a general hyperplane section, and let $\pi_* H$ denote its cycle-theoretic pushforward. By Lemma \ref{lem:Weil le Cartier}, we may choose a Cartier divisor $H_X \ge \pi_* H$ on $X$. We claim that there is a decomposition
\begin{sequation} \label{eqn:decomp}
\pi^* H_X = H + G + F
\end{sequation}where
\setitemize[1]{leftmargin=1.5em,parsep=0em,itemsep=0.125em,topsep=0.125em}
\begin{itemize}
\item the divisor $G$ is effective and contains no exceptional components, and
\item the divisor $F$ is effective $\pi$-exceptional and $\supp F = E_1 \cup \cdots \cup E_k$, where $E_1, \dots, E_k$ denote the codimension-one components of the $\pi$-exceptional set.
\end{itemize}
\setitemize[1]{leftmargin=*,parsep=0em,itemsep=0.125em,topsep=0.125em}
To get the decomposition (\ref{eqn:decomp}), set $G$ to be the strict transform $\pi^{-1}_{\;\;\,*}(H_X - \pi_* H)$, and set $F := \pi^* H_X - H - G$. It is clear that with these choices, the first point holds. The latter point holds because $\pi_* H$ contains the images of all the $E_i$, hence so does $H_X$.

Since $\supp E \subset \supp F$, we have $t F \ge E$ for sufficiently large $t \gg 0$. Therefore, we may define the following number:
\[ t_0 = \min \bigl\{ t \in \R \;\big|\; t F \ge E \bigr\} > 0. \]
Clearly $t_0 F - E \ge 0$, and by the choice of $t_0$, there is a component $E_i \not\subset \supp(t_0 F - E)$, so $(t_0 F - E)|_{E_i}$ is effective. Using (\ref{eqn:decomp}) to write
\[ -E = \underbrace{t_0 H}_{\text{$\pi$-ample}} + \underbrace{t_0 G + (t_0 F - E)}_{\text{effective, $\not\supset E_i$}} - \underbrace{\pi^*(t_0 H_X)}_{\text{$\pi$-trivial}} \]
and restricting to $E_i$, we see that $-E\big|_{E_i}$ is $\pi\big|_{E_i}$-big, as it is the sum of a $\pi\big|_{E_i}$-ample, an effective and a $\pi\big|_{E_i}$-trivial divisor. So the proposition holds with the choice of $E_0 := E_i$. \qed

\part{DIFFERENTIAL FORMS ON DLT PAIRS} \label{part:II}

\section{\texorpdfstring{\CC-pairs and \CC-differentials}{C-pairs and C-differentials}} \label{sec:C-pairs}

\emph{\CC-pairs}, first introduced by Campana \cite{Cam04} under the name of \emph{orbi\-foldes g\'eom\'etriques}, are a certain kind of pairs, which for example arise in the context of quotients by finite groups.
Here we will merely collect some definitions related to \CC-pairs which we use later.

\begin{dfn}[\CC-pair, \CC-multiplicity, see {\cite[Def.~2.2]{JK11}}] \label{dfn:C-pair}
A \emph{\CC-pair} is a pair $(X, D)$ where all the coefficients $a_i$ of $D = \sum_i a_i D_i$ are of the form $\frac{n_i-1}{n_i}$ for some $n_i \in \N^+ \cup \{ \infty \}$. We use the convention that $\frac{\infty-1}{\infty} := 1$.
The \emph{\CC-multiplicity} of $D_i$ in $(X, D)$ is defined to be $n_i$.
\end{dfn}

\begin{dfn}[Adapted morphism, see {\cite[Def.~2.7]{JK11}}] \label{dfn:adapted mor}
Let $(X, D)$ be a $\CC$-pair, with $D = \sum_i \frac{n_i-1}{n_i}D_i$. A
surjective morphism $\gamma: Y \to X$ from a normal variety
is called \emph{adapted} if the following holds.
\begin{enumerate}
\item For any number $i$ with $n_i < \infty$ and any irreducible divisor $E \subset Y$ that surjects onto $D_i$, the divisor $E$ appears in $\gamma^*(D_i)$ with multiplicity precisely $n_i$.
\item $\gamma$ is equidimensional.
\end{enumerate}
\end{dfn}

\begin{rem}[Pullback of Weil divisors]
In the setup of Definition \ref{dfn:adapted mor}, the pullback $\gamma^*(D_i)$ is defined without any \Q-Cartier assumption, because $\gamma$ is assumed to be equidimensional. This is explained in detail in \cite[Notation 2.5 and Convention 2.6]{JK11}.
\end{rem}

\begin{rem}[Existence of adapted morphisms]
Given a \CC-pair $(X, D)$, an adapted morphism exists by \cite[Prop.~2.9]{JK11}.
\end{rem}

\begin{ntn}
If $(X, D)$ is a $\CC$-pair, $\sigma \in \mathit\Gamma \bigl( X, \bigl( \Sym^{[m]} \Omega_X^p \bigr) (*\rp D.) \bigr)$ is any form, and $\gamma\!: Y \to X$ is a surjective morphism from a normal variety, then we denote by $\gamma^* \sigma \in \mathit\Gamma \bigl( Y, \bigl( \Sym^{[m]} \Omega_Y^p \bigr) (*\gamma^{-1} \rp D.) \bigr)$ the pullback of $\sigma$ as a differential form.
\end{ntn}

\begin{dfn}[\CC-differentials, see {\cite[Def.~3.5]{JK11}}] \label{dfn:C-diff}
If $(X, D)$ is a $\CC$-pair, we define the \emph{sheaves of \CC-differentials}
\[ \underbrace{ \Sym_{\CC}^{[m]} \Omega^p_X(\log D) }_{=:\sA } \subset
   \underbrace{ \bigl( \Sym^{[m]} \Omega_X^p \bigr) (*\rp D.)  }_{=: \sB} \]
on the level of presheaves as follows: if $U \subset X$ is open and
$\sigma \in \mathit\Gamma \bigl( U,\, \sB \bigr)$ any form, then $\sigma$ is a section of $\sA$ if and only if for any open
subset $U' \subset U$ and any adapted morphism $\gamma: V \to U'$, the
pull-back has at most logarithmic poles along $D_\gamma := \supp \gamma^* \rd D.$, and no
other poles elsewhere, i.e.
\[ \gamma^*(\sigma) \in \mathit\Gamma \bigl( V,\, \Sym^{[d]} \Omega^p_V(\log D_{\gamma} )\bigr). \]
\end{dfn}

\begin{rem}[see {\cite[Cor.~3.13, 3.14]{JK11}}]
The sheaves $\Sym_\CC^{[m]} \Omega_X^p(\log D)$ are reflexive, and we have inclusions
\[ \Sym^{[m]} \Omega_X^p(\log \rd D.) \subset \Sym_\CC^{[m]} \Omega_X^p(\log D) \subset \Sym^{[m]} \Omega_X^p(\log \rp D.). \]
For $m = 1$, the first inclusion is an equality. For $p = n = \dim X$, the definition boils down to
\begin{sequation} \label{eqn:souffl'e}
\Sym_\CC^{[m]} \Omega_X^n(\log D) = \mc O_X(mK_X + \rd mD.).
\end{sequation}
\end{rem}

Next we define the \emph{\CC-Kodaira--Iitaka dimension} of a sheaf of differentials. This notion makes sense on a \CC-pair $(X, D)$. It is a version of the usual Kodaira--Iitaka dimension which takes into account the fractional part of $D$.

\begin{dfn}[\CC-Kodaira--Iitaka dimension, see {\cite[Def.~4.3]{JK11}}] \label{dfn:kappa_C}
Let $(X, D)$ be a \CC-pair and let $\sA \subset \Sym_\CC^{[1]} \Omega_X^p(\log D)$ be a Weil divisorial subsheaf, for some number $p$. For any $m > 0$, there is a natural inclusion $\ms A^{[m]} \subset \Sym_\CC^{[m]} \Omega_X^p(\log D)$. Denote by $\Sym_\CC^{[m]} \ms A$ the saturation of $\ms A^{[m]}$ in $\Sym_\CC^{[m]} \Omega_X^p(\log D)$. Similarly to Definition \ref{dfn:kappa}, the sheaves $\Sym_\CC^{[m]} \ms A$ determine rational maps $\phi_m$ from $X$ to projective spaces if they have nonzero global sections. In this case, the \CC-Kodaira--Iitaka dimension $\kappa_\CC(\ms A)$ is defined to be $\max \big\{\!\dim \overline{\phi_m(X)}\big\}$. Otherwise, we set $\kappa_\CC(\ms A) = -\infty$.
\end{dfn}

\begin{rem} \label{rem:kappa_C ge kappa}
In the situation of Definition \ref{dfn:kappa_C}, we always have the obvious inequality $\kappa_\CC(\ms A) \ge \kappa(\ms A)$.
\end{rem}

\begin{rem}
Unlike the usual Kodaira dimension, $\kappa_\CC(\ms A)$ does not only depend on the sheaf $\ms A$, but also on the embedding of $\ms A$ as a subsheaf of $\Sym_\CC^{[1]} \Omega_X^p(\log D)$ and on the fractional part of $D$.
In order to emphasize the dependence on the fractional part of $D$, we often write $\Sym_\CC^{[1]} \Omega_X^p(\log D)$ instead of $\Omega_X^{[p]}(\log \rd D.)$, although technically there is no difference.
\end{rem}

\section{\texorpdfstring{Adjunction and residues on dlt \CC-pairs}{Adjunction and residues on dlt C-pairs}} \label{sec:dlt adjunction}

\subsection{The residue map for snc pairs}

One of the most important features of logarithmic differential forms on snc
pairs $(X, D)$ is the existence of the \emph{residue sequence}
\begin{sequation}\label{seq:snc res map}
  0 \to \Omega_X^p ( \log D - D_0 ) \to \Omega_X^p(\log D)
  \xrightarrow{\;\res_{D_0}\;} \Omega_{D_0}^{p-1}\bigl(\log
  (D-D_0)\big|_{D_0}\bigr) \to 0
\end{sequation}%
for any integer $p$ and any irreducible component $D_0 \subset D$
\cite[2.3.b]{EV92}. Residues have manifold applications in Algebraic
Geometry. Here we list just a few.
\begin{itemize}
\item The adjunction formula $(K_X + D_0)\big|_{D_0} = K_{D_0}$ for a smooth
  divisor $D_0$ in a smooth variety $X$ can be proven by restricting sequence
  (\ref{seq:snc res map}) to $D_0$ for the special values $D = D_0$ and $p =
  \dim X$. In fact, it is precisely this connection between adjunction and
  residues which we will pursue in this section.
\item In favorable cases, the residue sequence allows for induction on the
  dimension. For example, Esnault and Viehweg \cite[Cor.~6.4]{EV92} employed
  this scheme to deduce Kodaira--Akizuki--Nakano vanishing from Deligne's Hodge
  theory and the topological Andreotti--Frankel vanishing theorem.
\item In \cite[Part VI]{GKKP11}, the residue map was used as a test for
  logarithmic poles: if a form in $\Omega_X^p(\log D)$ has zero residue along
  $D_0$, then it does not have a pole along $D_0$.
\end{itemize}

\subsection{The restriction map for snc pairs}

In a similar vein, recall from \cite[2.3.c]{EV92} that there exists a
restriction map for logarithmic differentials,
\begin{sequation}\label{eqn:snc restr map}
  0 \to \Omega_X^p(\log D)(-D_0) \to \Omega_X^p ( \log D - D_0 )
  \xrightarrow{\restr_{D_0}} \Omega_{D_0}^p\bigl(\log (D-D_0)|_{D_0}\bigr) \to 0.
\end{sequation}%

\subsection{Main result of this section}

One cannot expect to have sequences like (\ref{seq:snc res map}) and
\eqref{eqn:snc restr map} for reflexive differentials on singular
pairs in general. However, in \cite[Thm.~11.7]{GKKP11} it was shown
that as long as $(X, D)$ is dlt, a three-term residue sequence for
$\Omega_X^{[p]}(\log \rd D.)$ exists and is exact in codimension
two. In Theorem \ref{thm:res of symm diff} we will extend this result
by showing that in the setting of dlt \CC-pairs, residue and
restriction maps even exist for symmetric \CC-differentials, if one
endows the boundary divisor $D_0$ with a suitable \CC-pair structure
using the \emph{different}.

\subsubsection{Definition of the different}

The different is a divisor on $D_0$ which has a clear geometric meaning, namely it makes the adjunction formula work in the singular setting.

\begin{prpdfn}[Different, see {\cite[Prop./Def.~16.5]{Kol92}} or {\cite[Sec.~4]{Kol11}}] \label{prpdfn:different}
Let $(X, D)$ be a pair, with $D$ a $\Q$-divisor, and let $S \subset X$ be a normal codimension-one subvariety which is not contained in $\supp(D)$. Assume that $K_X + S + D$ is \Q-Cartier at every codimension-one point of $S$. Then there is a canonically defined $\Q$-divisor on $S$, called the \emph{different} $\Diff_S(D)$, with the property that
\begin{sequation} \label{eqn:different}
(K_X + S + D)|_S \,\sim_\Q\, K_S + \Diff_S(D).
\end{sequation}%
If $(X, S + D)$ is log canonical, then $\Diff_S(D)$ is an effective divisor on $S$.
If $(X, S + D)$ is dlt, then the assumption on $K_X + S + D$ being \Q-Cartier at every codimension-one point of $S$ is automatically satisfied. \qed
\end{prpdfn}

\begin{rem}
Note that the different is an honest divisor, even though (\ref{eqn:different}) only determines it up to \Q-linear equivalence.
\end{rem}

We give some examples of differents.

\begin{exm}[SNC pairs]\label{exm:516}
In the setting of Proposition/Definition \ref{prpdfn:different}, suppose that $(X, S)$ is snc (or more generally, that $S$ is Cartier in codimension two). Then $\Diff_S(D) = D|_S$, since $(K_X + S)|_S = K_S$.
\end{exm}

\begin{exm}[Quadric cone]\label{exm:491}
  Let $X \subset \P^3$ be the projective quadric cone with vertex $P$, and $L$ a
  ruling of $X$. Then 
  $$
  \omega_X = \O_X(-2), \omega_L = \O_L(-2P) \text{ and } \O_X(2L) = \O_X(1).
  $$
  Hence $\Diff_L(0) = \frac12 P$.
\end{exm}

\begin{exm}[DLT pairs]\label{exm:547}
  Let $X$ be a normal variety, let $D_0, D_1$ be distinct prime divisors, and
  let $n_1 \in \N^+ \cup \{ \infty \}$. Assume that the pair $(X, D_0 +
  \frac{n_1 - 1}{n_1} D_1)$ is dlt. Recall from \cite[Prop.~5.51]{KM98} that
  $D_0$ is normal. The different $\Diff_{D_0}(D_1)$ is thus defined. If $W$
  is any prime divisor on $D_0$, then an elementary but tedious calculation
  in the spirit of \cite[Prop.~16.6.3]{Kol92} shows that the coefficient $w$ of
  $W$ in $\Diff_{D_0}(D_1)$ is given as follows. Firstly, let $m$ be the index
  of $K_X + D_0$ at the general point of $W$. For the definition of index, see
  Definition \ref{dfn:index} below. Secondly, if $W \subset D_1$, let $n = n_1$,
  and if $W \not\subset D_1$, set $n = 1$. Then, using the convention that $m
  \cdot \infty = \infty$, we have
\[ w = 1 - \frac1{mn}. \]
\end{exm}

Note that Example \ref{exm:491} is a special case of Example \ref{exm:547}.

\subsubsection{Statement of main result of this section}

Using Proposition/Definition \ref{prpdfn:different}, we can formulate the theorem announced above.

\begin{thm}[Residues and restrictions of symmetric differentials]\label{thm:res of symm diff} 
  Let $(X, D)$ be a dlt \CC-pair and $D_0 \subset \rd D.$ a component of the
  reduced boundary. Recall from \cite[Cor.~5.52]{KM98} that $D_0$ is normal,
  and set $D_0^c := \Diff_{D_0}(D - D_0)$. Then the pair $(D_0, D_0^c)$ is also
  a dlt \CC-pair, and the following holds.
  \begin{enumerate}
  \item\label{itm:6.12.1} For any integer $p \ge 1$, there is a map
    \[ 
    \res_{D_0}^k\!: \Sym_\CC^{[k]} \Omega_X^p(\log D) \to \Sym_\CC^{[k]}
    \Omega_{D_0}^{p-1}(\log D_0^c)
    \]
    which generically coincides with the $k$-th symmetric power of the residue
    map \eqref{seq:snc res map} in the following sense. More precisely, setting
    \[
    U := (X, \rp D.)_\snc \minus \supp \{ D \}, \; D^\circ := D \cap U, \;
    \text{and } D_0^\circ := D_0 \cap U, 
    \]
    and recalling from Example \ref{exm:516} that $D_0^c|_{D_0^\circ} = (D^\circ
    - D_0^\circ)|_{D_0^\circ}$, we have a commutative diagram
    \[ 
    \xymatrix{
      \Sym_\CC^{[k]} \Omega_X^p(\log D)\big|_U \ar[rr]^-{\res_{D_0}^k\bigr|_U} \ar@{=}[d] & & \Sym_\CC^{[k]} \Omega_{D_0}^{p-1}(\log D_0^c)\big|_{D_0^\circ} \ar@{=}[d] \\
      \Sym^{k} \Omega_U^p(\log D^\circ) \ar[rr]^-{\Sym^k \!\res_{D_0^\circ}} & & \Sym^{k} \Omega_{D_0^\circ}^{p-1}\bigl(\log (D^\circ - D_0^\circ)\big|_{D_0^\circ}\bigl). \\
    }
    \]
  \item\label{itm:6.12.2} For any integer $p \ge 0$, there is a map
    \[
    \restr_{D_0}^k\!: \Sym_\CC^{[k]} \Omega_X^p(\log D - D_0) \to \Sym_\CC^{[k]} \Omega_{D_0}^p(\log D_0^c) 
    \]
    which, in the same sense as above, generically coincides with the $k$-th
    symmetric power of the restriction map~\eqref{eqn:snc restr map}.
  \end{enumerate}
\end{thm}

We will prove Theorem \ref{thm:res of symm diff} in Section \ref{subsec:652}, after some auxiliary results.


\subsection{The codimension-two structure of dlt pairs}

The proof of Theorem \ref{thm:res of symm diff} draws on an analysis
of the local structure of dlt pairs. Proposition \ref{prp:dlt quot}
asserts that along the reduced boundary, dlt pairs have quotient
singularities in codimension two.  Before we can state this
proposition, we need an auxiliary definition.

\begin{dfn}[Index] \label{dfn:index}
Let $X$ be a normal variety and $D$ an integral Weil divisor on $X$. For any point $P \in X$, we denote by $i_D(P)$ or by $i(P, D)$ the \emph{index} of $D$ in $P$, i.e.
\[ i_D(P) := i(P, D) := \inf \{ \, n \in \N^+ \;|\; \text{$nD$ is Cartier at $P$} \, \} \in \N^+ \cup \{ \infty \}. \]
\end{dfn}

\begin{lem}[Semicontinuity of index] \label{lem:index semicont}
In the setting of Definition \ref{dfn:index}, the function
\[ \begin{array}{cccl}
i_D\!: & X & \to     & \N^+ \cup \{ \infty \}, \\
       & P & \mapsto & i_D(P),
\end{array} \]
is upper semicontinuous.
\end{lem}

\begin{proof}
If $i_D(P) = n < \infty$, then $nD$ is Cartier in a neighborhood $U$ of $P$, hence $i_D(Q) \le n$ for all $Q \in U$.
\end{proof}

\begin{figure}[t]
  \centering
  \[
  \xymatrix{
    *++{\begin{tikzpicture}(4,4)(0,0)
        \draw (0,0) to [out=90, in=180] (2,2);
        \draw (2,2) to [out=0, in=90] (6,0);
        \fill (4.75,1.8) node[right]{\scriptsize $X$};
        \draw (6,0) to [out=-90, in=-30] (3,-1);
        \draw (3,-1) to [out=150, in=-90] (0,0);
        \draw[line width = .1, fill=black!20] (4.1,.35) circle (.5);
        \draw (1,.25) to [out=20, in=160] (5,0); 
        \draw (3.9,-.9) to [out=80, in=-90] (4.1,1.5);
        \fill (4.85,.02) node[below]{\scriptsize $\rd D.$};
        \fill (4,-.9) node[above left]{\scriptsize $\rd D.$};
        \draw[dashed] (2.5,-.55) to [out=80, in=-80] (2.5,1.65); 
        \fill (2.5,1.45) node[right]{\scriptsize $\{ D \}$};
        \draw[line width = .1] (1.3,.34) circle (.5);
        \draw[line width = .1] (2,.5) circle (.5);
        \draw[line width = .1] (2.7,.55) circle (.5);
        \draw[line width = .1] (3.4,.5) circle (.5);
        \fill (4.3,.6) node[above right]{\scriptsize $U_\alpha$};
        \draw[line width = .1] (4.8,-.02) circle (.5);
        \draw[line width = .1] (3.8,-.5) circle (.5);
        \draw[line width = .1] (3.95,1.2) circle (.5);
    \end{tikzpicture}
    }
    & & &
    *++{\begin{tikzpicture}(4,4)(0,0)
        \draw (0,.8) to [out=90, in=180] (1.5,1.8);
        \draw (1.5,1.8) to [out=0, in=90] (3,.8);
        \draw (3,.8) to [out=-90, in=-30] (2,-.2);
        \draw (2,-.2) to [out=150, in=-90] (0,.8);
        \fill (2.3,1.8) node[right]{\scriptsize $V_\alpha$};
        \draw (0,.8) to [out=20, in=160] (2.98,.6);
        \draw (1.5,-.01) to [out=80, in=-80] (1.5,1.8);
        \fill (1.52,.83) node[below right]{\scriptsize $\gamma_\alpha^* \rd D.$};
    \end{tikzpicture}
    }
    \ar[lll]_-{\txt<6pc>{\scriptsize $\gamma_\alpha$, branches only along codim 2 sets}}
      ^-{\text{contained in $\rd D.$}}
  }
  \]
  \caption{Situation of Proposition \ref{prp:dlt quot}} \label{fig:dlt quot}
\end{figure}
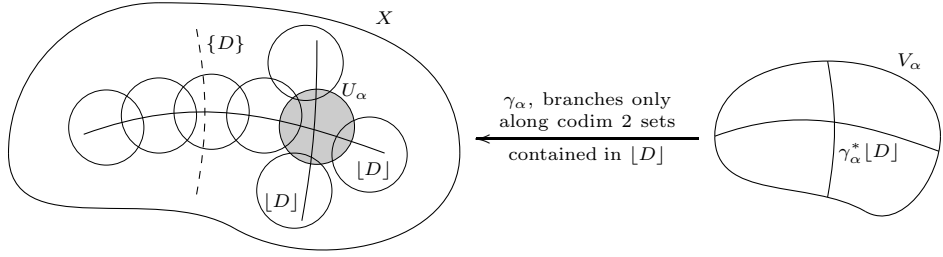

\begin{prp}[Codimension-two structure of dlt pairs along the reduced boundary] \label{prp:dlt quot}
Let $(X, D)$ be a dlt pair with $\rd D. \ne 0$. Then there exists a closed subset $Z \subset X$ with $\codim_X Z \geq 3$ such that $X \minus Z$ is \Q-factorial, and a covering of $\supp \rd D. \minus Z$ by finitely many Zariski-open subsets $(U_\alpha)_{\alpha \in I}$ of $X \minus Z$ which admit finite cyclic Galois covering maps
\[ \gamma_\alpha\!: V_\alpha \to U_\alpha \]
satisfying the following properties for all $\alpha \in I$.
\begin{enumerate}
\item\label{itm:dlt quot.1} The pair $\bigl( V_\alpha, \gamma_\alpha^*\rd D. \bigr)$ is quasi-projective, reduced, and snc.
\item\label{itm:dlt quot.2} If $P \in U_\alpha$ is any point such that $i(P, K_X + \rd D.) = 1$, then $\gamma_\alpha$ is \'etale over $P$. In particular, $\gamma_\alpha$ is \'etale in codimension one.
\item\label{itm:dlt quot.3} If $P \in U_\alpha$ is any point such that $i(P, K_X + \rd D.) > 1$, then $\deg \gamma_\alpha = i(P, K_X + \rd D.)$ and $\gamma_\alpha$ is totally branched over $P$.
\item\label{itm:dlt quot.4} The branch locus of $\gamma_\alpha$ is contained in $\supp \rd D.$.
\end{enumerate}
\end{prp}

For a visualization of the statement of Proposition \ref{prp:dlt quot}, see Figure \ref{fig:dlt quot}.

\begin{rem} \label{rem:554}
Given any index $\alpha$, items (\ref{prp:dlt quot}.\ref{itm:dlt quot.1})--(\ref{prp:dlt quot}.\ref{itm:dlt quot.3}) together imply that the branch locus of each $\gamma_\alpha$ is exactly the singular locus $U_{\alpha, \sg}$. Furthermore, the index of $K_X + \rd D.$ is constant along each of the sets $U_{\alpha, \sg}$.
Item (\ref{prp:dlt quot}.\ref{itm:dlt quot.4}) then implies that $U_{\alpha, \sg} \subset \supp \rd D.$.
\end{rem}

\begin{warn}
In Proposition \ref{prp:dlt quot}, we only claim that $\supp \rd D. \minus Z \subset \bigcup_{\alpha \in I} U_\alpha$. We do not claim that the $U_\alpha$ cover $X$ or $X \minus Z$.
\end{warn}

\subsubsection{Proof of Proposition \ref{prp:dlt quot}}

We start by removing some codimension-three subsets from $X$. This allows us to make additional simplifying assumptions. First of all, by \cite[Prop.~9.1]{GKKP11} a dlt pair is always $\Q$-factorial in codimension two, so we may make the following assumption.
\begin{awlog} \label{awlog:simplif0}
$X$ is $\Q$-factorial. In particular $i(P, K_X + \rd D.) < \infty$ for all $P \in X$.
\end{awlog}
Next, we may remove the irreducible components of $X_\sg$ which have codimension $\ge 3$ in $X$, as well as those that are not contained in $\supp \rd D.$. Furthermore, removing the singular locus of $X_\sg$, we may assume that $X_\sg$ itself is smooth. Covering the remaining components of $X_\sg$ one at a time, we may thus assume the following.
\begin{awlog} \label{awlog:simplif}
$X_\sg$ is irreducible, of codimension two, and contained in $\supp \rd D.$.
\end{awlog}
Consider $i_0$, the minimum of the function $P \mapsto i(P, K_X + \rd D.)$ along $X_\sg$, and the set
\[ U = \bigl\{ P \in X \;\big|\; i(P, K_X + \rd D.) \le i_0 \bigr\}, \]
which is open by Lemma \ref{lem:index semicont}. It is clear that on $X_\sg \cap U$, the function $i(\cdot, K_X + \rd D.)$ is constant with value $i_0$. Note that $X_\sg \minus U$, being a proper closed subset of $X_\sg$, has codimension $\ge 3$ in $X$. Hence removing it, we may additionally assume the following.
\begin{awlog} \label{awlog:simplif2}
The index of $K_X + \rd D.$ is constant along $X_\sg$.
\end{awlog}

Now we apply \cite[Cor.~9.14]{GKKP11} to the pair $(X, D)$, which gives us the desired maps $\gamma_\alpha$. It remains to check that these have the required properties. For this, note that the $\gamma_\alpha$ are constructed as local index one covers for $K_X + \rd D.$. So it is immediate that they are cyclic Galois and that (\ref{prp:dlt quot}.\ref{itm:dlt quot.2}) holds. Claim (\ref{prp:dlt quot}.\ref{itm:dlt quot.1}), as well as the finiteness of the index set $I$, is then a consequence of \cite[Cor.~9.14.1]{GKKP11}.

To obtain (\ref{prp:dlt quot}.\ref{itm:dlt quot.3}), note that by the additional assumption \eqref{awlog:simplif2} on the index of $K_X + \rd D.$, a local index one cover $\gamma_\alpha\!: V_\alpha \to U_\alpha$ for $K_X + \rd D.$ with respect to a point $P \in X_\sg$ is also an index one cover with respect to any other point $Q \in X_\sg \cap U_\alpha$. It follows by \cite[Prop.~in (3.6) on p.~362]{Rei87} that the $\gamma_\alpha$-preimage of any point in $X_\sg \cap U_\alpha$ consists of a single point.

As for (\ref{prp:dlt quot}.\ref{itm:dlt quot.4}), the branch locus of $\gamma_\alpha$ is contained in $X_\sg$, and $X_\sg \subset \supp \rd D.$ by Assumption (\ref{awlog:simplif}). \qed

\subsection{Adjunction on dlt \CC-pairs}

We show that the class of dlt \CC-pairs is ``stable under adjunction'', and that along the reduced boundary, a dlt \CC-pair can be covered by adapted morphisms in a way compatible with the adjunction formula.

The following two propositions will be shown in Sections \ref{subsubsec:prf dlt adjunction} and \ref{subsubsec:prf dlt adapt}, respectively.

\begin{prp}[Adjunction for dlt \CC-pairs] \label{prp:dlt adjunction}
Let $(X, D)$ be a dlt \CC-pair, and let $D_0 \subset \rd D.$ be a component of the reduced boundary. Then:
\begin{enumerate}
\item\label{itm:6.8.1} The divisor $D_0$ is normal.
\item\label{itm:6.8.1+eps/2} Setting $D_0^c := \Diff_{D_0}(D - D_0)$, then $(D_0, D_0^c)$ is a dlt \CC-pair.
\item\label{itm:6.8.1+eps} If $W$ is any irreducible component of the non-snc locus of $(X, \rp D.)$ such that $\codim_X W = 2$ and $W$ is contained in $D_0$, then $W \subset \supp \rp D_0^c.$.
\item\label{itm:6.8.2} We have $\bigl(\rd D. - D_0\bigr)\big|_{D_0} = \rd D_0^c.$.
\end{enumerate}
\end{prp}

An instance of (\ref{prp:dlt adjunction}.\ref{itm:6.8.1+eps}) is given by Example \ref{exm:491}, where the set $W$ is exactly the vertex of the cone.

\begin{prp}[Adapted morphisms for dlt pairs along the reduced boundary] \label{prp:dlt adapt}
Let $(X, D)$ be a dlt \CC-pair, and let $D_0 \subset \rd D.$ be a component of the reduced boundary. Then there exists a closed subset $Z \subset X$ with $\codim_X Z \geq 3$ and a covering of $D_0 \minus Z$ by Zariski-open subsets $(U_j)_{j \in J}$ of $X \minus Z$ which admit finite covering maps
\[ \delta_j\!: W_j \to U_j \]
satisfying the following properties for all $j \in J$.
\begin{enumerate}
\item\label{itm:dlt adapt.1} The pair $(W_j, \delta_j^* \rd D.)$ is quasi-projective, reduced, and snc.
\item\label{itm:dlt adapt.2} The morphism $\delta_j$ is adapted for the \CC-pair $(U_j, D|_{U_j})$.
\item\label{itm:dlt adapt.3a} The divisor $E_{0, j} := \delta_j^{-1}(D_0)$ is smooth and irreducible.
\item\label{itm:dlt adapt.3b} The restricted morphism $\delta_j|_{E_{0, j}}\!: E_{0, j} \to D_0 \cap U_j$ is adapted for the \CC-pair $(D_0 \cap U_j, D_0^c|_{D_0 \cap U_j})$.
\end{enumerate}
\end{prp}

\subsubsection{Preparations for the proofs of Propositions \ref{prp:dlt adjunction} and \ref{prp:dlt adapt}}

We will prove three auxiliary lemmas.
The first two of them concern the question of how the boundary components of a \CC-pair intersect, if the pair is additionally assumed to be dlt. The argument relies on the trivial observation that the coefficients of the boundary divisor cannot be very small: they are at least $1/2$.~--- The last lemma is about the existence of adapted morphisms.

\begin{lem} \label{lem:three comp}
Let $(X, D)$ be a dlt \CC-pair, and let $D_0, D_1, D_2$ be three components of $D$, one of which is reduced. Then $\codim_X (D_0 \cap D_1 \cap D_2) \ge 3$.
\end{lem}

\begin{proof}
We may assume that $D_0$ is reduced. Cutting with hyperplanes \cite[Lem.~2.25]{GKKP11}, we reduce to the case where $X$ is a surface. Then by \cite[Prop.~4.11]{KM98}, $X$ is \Q-factorial. This implies that $(X, D_0 + \frac12 D_1 + \frac12 D_2)$ is still a dlt pair, by \cite[Cor.~2.39]{KM98}. Thus we may assume that $D = D_0 + \frac12 D_1 + \frac12 D_2$. In particular, we assume that $(X, D)$ is plt.

Proceeding by contradiction, assume that $D_0, D_1, D_2$ intersect in a point $P$. In a neighborhood $P \in U \subset X$, take an index one cover $p\!: Y \to U$ at $P$ with respect to $K_X + D_0$, and let $Q$ be the unique point lying over $P$. Then $(Y, p^* D)$ is plt and $(Y, p^* D_0)$ is canonical, hence $Y$ is smooth at $Q$ by \cite[Thm.~4.5]{KM98}. Blowing up $Q$, we see that $\discrep(Y, p^* D) = 1 - \mult_Q(p^* D) \le -1$, a contradiction to $(Y, p^* D)$ being plt.
\end{proof}

\begin{lem} \label{lem:smooth in codim 2}
Let $(X, D)$ be a dlt \CC-pair, where $X$ is smooth. Let $D_0$ be a reduced component of $D$. Then $\codim_X \bigl( D_0 \cap (X, \rp D.)_\nsnc \bigr) \ge 3$.
\end{lem}

\begin{proof}
Using \cite[Lemmas 2.23.5 and 2.25]{GKKP11}, we again reduce to the surface case by cutting with hyperplanes. By Lemma \ref{lem:three comp}, we may assume that $D$ has only two components. Rounding down if necessary, we may assume that $D = D_0 + \frac12 D_1$. In particular, $(X, D)$ is plt.
Let $P$ be a point of intersection of $D_0$ and $D_1$. We will first show that $D_0$ and $D_1$ are both smooth at $P$. Secondly, we will show that they intersect transversely.

Being the reduced boundary of a plt pair, $D_0$ is normal \cite[Prop.~5.51]{KM98}. Since it is a curve, it is even smooth. Now proceeding by contradiction, we assume that $D_1$ is singular at $P$. Blowing up at $P$, we see that $\discrep(X, D) = 1 - \mult_P(D) \le -1$, a contradiction to $(X, D)$ being plt. Hence $D_0$ and $D_1$ are both smooth.

If $D_0$ and $D_1$ intersect non-transversely in $P$, we first blow up at $P$. Denote the strict transforms of $D_0$ and $D_1$ by $\tilde D_0$ and $\tilde D_1$, respectively, and let $E$ be the exceptional divisor. Blowing up the point $\tilde D_0 \cap \tilde D_1 \cap E$, we again obtain a contradiction to $(X, D)$ being plt.
\end{proof}

\begin{lem}[Existence of adapted morphisms] \label{lem:ex adapt}
Let $(X, D = D_0 + \frac{n_1 - 1}{n_1} D_1)$ be a quasi-projective \CC-pair, with $(X, \rp D.)$ snc, $n_1 < \infty$ and $D_0, D_1$ irreducible. Assume that $D_1$ is a principal divisor and that $D_0 \cap D_1 \ne \emptyset$. Then there is a finite cover $\beta\!: Y \to X$ with the following properties.
\begin{enumerate}
\item\label{itm:ex adapt.1} The pair $(Y, \beta^* D_0)$ is snc and $\beta^* D_0$ is irreducible.
\item\label{itm:ex adapt.2} The morphism $\beta$ is adapted for the pair $(X, D)$.
\item\label{itm:ex adapt.3} The branch locus of $\beta$ is exactly $D_1$, and $\beta$ is totally branched over $D_1$.
\end{enumerate}
\end{lem}

\begin{proof}
Let $D_1$ be given by $\{ f = 0 \}$ for some function $f$ on $X$. Consider $X \x \C$ with coordinate $z$ in the second factor, and define $Y := \{ z^{n_1} = f \} \subset X \x \C$ with $\beta$ the restriction of the projection to the first factor. By a calculation in local coordinates, both $Y$ and $\beta^* D_0$ are smooth. It is then clear that (\ref{lem:ex adapt}.\ref{itm:ex adapt.2}) and (\ref{lem:ex adapt}.\ref{itm:ex adapt.3}) hold.

Assume that $\beta^* D_0$ is reducible, containing two distinct components $T$ and $S$. Since $\beta$ is totally branched over $D_1$, the components $T$ and $S$ must meet in any point of $\beta^{-1}(D_0 \cap D_1)$. This contradicts the smoothness of $\beta^* D_0$, so $\beta^* D_0$ is irreducible. This yields (\ref{lem:ex adapt}.\ref{itm:ex adapt.1}).
\end{proof}

\subsubsection{Proof of Proposition \ref{prp:dlt adjunction}}\label{subsubsec:prf dlt adjunction}

As before, normality of $D_0$ follows from \cite[Corollary 5.52]{KM98}. Easy adjunction \cite[Thm.~17.2]{Kol92} implies that $(D_0, D_0^c)$ is dlt. By Lemma \ref{lem:three comp}, we see that we can determine the coefficients of $D_0^c$ as in Example \ref{exm:547}. In particular, $(D_0, D_0^c)$ is a \CC-pair. This is (\ref{prp:dlt adjunction}.\ref{itm:6.8.1+eps/2}).

To show (\ref{prp:dlt adjunction}.\ref{itm:6.8.1+eps}), we argue by contradiction and assume there is a codimension-two component $W$ of the non-snc locus of $(X, \rp D.)$ which is contained in $D_0$ but not in $\supp \rp D_0^c.$. In the notation of Example \ref{exm:547}, this says that $w = 0$, which can only happen if $m = n = 1$. This means that at the general point $\xi$ of $W$, $D = D_0$ and $K_X + D$ is Cartier. By \cite[Prop.~9.12]{GKKP11}, the pair $(X, D = \rp D.)$ is plt at $\xi$. In fact, $(X, D)$ is even canonical at $\xi$, since $K_X + D$ is Cartier at $\xi$. But then $(X, D)$ is snc at $\xi$ by \cite[Prop.~9.2]{GKKP11}. So we obtain a contradiction to the definition of $W$.

We turn to (\ref{prp:dlt adjunction}.\ref{itm:6.8.2}). Again, let $W \subset D_0$ be a prime divisor and let $w$ be the coefficient of $D_0^c$ along $W$. Looking at Example \ref{exm:547} and noting that $m < \infty$, we see that
\[ W \subset \rd D_0^c. \quad \Longleftrightarrow \quad \rd w. = 1 \quad \Longleftrightarrow \quad n = \infty \quad \Longleftrightarrow \quad W \subset \rd D. - D_0. \]
This finishes the proof. \qed

\subsubsection{Proof of Proposition \ref{prp:dlt adapt}} \label{subsubsec:prf dlt adapt}

The proof is divided into five steps. First, we explain the general idea. Then we make some simplifying assumptions, without loss of generality. Finally, we construct the required covering maps in two steps, and we prove that they satisfy the required properties.

\subsubsection*{Step 0: Outline of proof}

Fix any point $x \in X$. Since we only want to cover $D_0$ by the images of adapted morphisms, we may assume $x \in D_0$. Assuming $x$ is not contained in a ``bad'' set of codimension three, Proposition \ref{prp:dlt quot} allows to find a neighborhood $x \in U_\alpha$ and a finite map $\gamma_\alpha\!: V_\alpha \to U_\alpha \subset X$ satisfying (\ref{prp:dlt quot}.\ref{itm:dlt quot.1})--(\ref{prp:dlt quot}.\ref{itm:dlt quot.4}). Since this map is \'etale in codimension one, it clearly cannot be adapted for $(X, D)$. Therefore, after shrinking $U_\alpha$ to a smaller neighborhood $x \in U$ and setting $V := \gamma_\alpha^{-1}(U)$, $\gamma := \gamma_\alpha|_V$, we use Lemma \ref{lem:ex adapt} to construct an adapted morphism $\beta\!: W \to V$ for the \CC-pair $(V, \gamma^* D)$. The desired map $\delta\!: W \to U$ is then given as $\gamma \circ \beta$. Since the point $x$ was chosen arbitrarily, this ends the proof.~--- The diagram in Figure~\ref{fig:dlt adapt} summarizes the construction.
\begin{figure}
\centering
\[ \xymatrix{
E_0 \ar@{=}[rr] \ar@/_/[rrdd]_-{\delta|_{E_0}}
& & \delta^{-1}(D_0) \ar@{ ir->}[rrr]^{\quad\text{inclusion}} \ar[d] & & & W \ar[d]^-{\beta \text{, adapted}} \ar@/_1.5pc/[dd]|\hole_(.35){\delta} \\
& & \gamma^{-1}(D_0) \ar@{ ir->}[rrr]^{\quad\text{inclusion}} \ar[d] & & & V \ar[d]^-{\gamma} \ar@{ ir->}[r] & V_\alpha \ar[d]^-{\gamma_\alpha \text{, \'etale in codim 1}} \\
& & D_0 \cap U \ar@{ ir->}[rrr]^{\quad\text{inclusion}}              & & & U \ar@{ ir->}[r] & U_\alpha \ar@{ ir->}[r] & X \\
} \]
\caption{Construction used in the proof of Proposition~\ref{prp:dlt adapt}}
\label{fig:dlt adapt}
\end{figure}
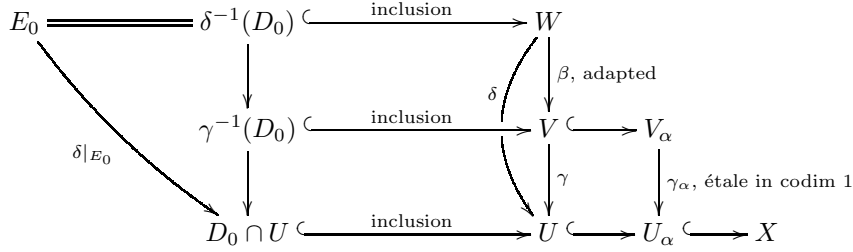

\subsubsection*{Step 1: Simplifying assumptions}

Removing finitely many closed subsets $S \subset X$ with the property that $\codim_X (S \cap D_0) \ge 3$, we can make additional simplifying assumptions.
First, by \cite[Prop. 9.12]{GKKP11}, we may assume the following.
\begin{awlog} \label{awlog:dichotomy}
For every point $x \in D_0$, one of the following two conditions holds.
\begin{enumerate}
\item\label{itm:dich1} The pair $(X, D)$ is snc, $D$ is reduced around $x$, and there are exactly two components of $\rd D.$ passing through $x$.
\item\label{itm:dich2} The pair $(X, D)$ is plt at $x$, and $\rd D.$ is smooth at $x$ (but $X$ need not be smooth at $x$).
\end{enumerate}
\end{awlog}

By Lemma \ref{lem:smooth in codim 2}, we may furthermore assume the following.

\begin{awlog} \label{awlog:701}
At any point $x \in D_0$ where $X$ is smooth, the pair $(X, \rp D.)$ is snc.
\end{awlog}

By Lemma \ref{lem:three comp} we may remove the locus where $D_0$ and at least two other components of $D$ intersect.
\begin{awlog} \label{awlog:690}
Through any point $x \in D_0$, there passes at most one component of $D - D_0$.
\end{awlog}

We apply Proposition \ref{prp:dlt quot} to the pair $(X, D_0)$.
\begin{awlog} \label{awlog:712}
The divisor $D_0$ is covered by finitely many open subsets $U_\alpha \subset X$, $\alpha \in I$, which admit covering maps $\gamma_\alpha\!: V_\alpha \to U_\alpha$ as in Proposition \ref{prp:dlt quot}.
\end{awlog}

Let $\alpha \in I$ be any index and $x \in D_0 \cap U_\alpha$ any point. We consider the dichotomy given by Assumption \ref{awlog:dichotomy}. If (\ref{awlog:dichotomy}.\ref{itm:dich1}) holds at $x$, then $\gamma_\alpha$ is \'etale over $x$ by (\ref{prp:dlt quot}.\ref{itm:dlt quot.2}). Hence $(V_\alpha, \rp \gamma_\alpha^* D.)$ is snc at all points of $\gamma_\alpha^{-1}(x)$.
If (\ref{awlog:dichotomy}.\ref{itm:dich2}) holds at $x$, then since $\gamma_\alpha$ is \'etale in codimension one by (\ref{prp:dlt quot}.\ref{itm:dlt quot.2}), we have that $(V_\alpha, \gamma_\alpha^* D)$ is a plt \CC-pair. Now since $V_\alpha$ is smooth by (\ref{prp:dlt quot}.\ref{itm:dlt quot.1}), Lemma \ref{lem:smooth in codim 2} together with the finiteness of the index set $I$ leads to the following extra assumption.
\begin{awlog} \label{awlog:697}
For any index $\alpha$, the pair $(V_\alpha, \rp \gamma_\alpha^* D.)$ is snc.
\end{awlog}

Under Assumption \ref{awlog:697}, the set where at least three components of $\rp \gamma_\alpha^* D.$ meet has codimension $\ge 3$ in $V_\alpha$. Removing the closure of its image in $X$, we can make the following assumption.

\begin{awlog} \label{awlog:723}
For any index $\alpha$ and any point $y$ of $V_\alpha$, there are at most two components of $\rp \gamma_\alpha^* D.$ passing through $y$.
\end{awlog}

\subsubsection*{Step 2: Construction of $U$ and $\gamma$}

We formulate the content of this step in the following claim.

\begin{clm} \label{clm:step 2}
For any point $x \in D_0$, there exists a neighborhood $x \in U \subset X$ and a finite map $\gamma\!: V \to U$, \'etale in codimension one, such that the following properties hold.
\begin{enumerate}
\item\label{itm:step 2.1} The pair $(V, \gamma^* \rp D.)$ is snc.
\item\label{itm:step 2.2} Over any point $P \in U$ where $i(P, K_X + D_0) = 1$, the map $\gamma$ is \'etale.
\item\label{itm:step 2.3} Over any point $P \in U$ where $i(P, K_X + D_0) > 1$, the map $\gamma$ is totally ramified of order $i(P, K_X + D_0)$.
\item\label{itm:step 2.4} The divisor $\gamma^* D_0$ is irreducible.
\item\label{itm:step 2.5} Writing $\gamma^* D = \gamma^* D_0 + D_V$, we have that $\supp D_V$ is either empty or irreducible, and that $\rp D_V.$ is a principal divisor.
\end{enumerate}
\end{clm}

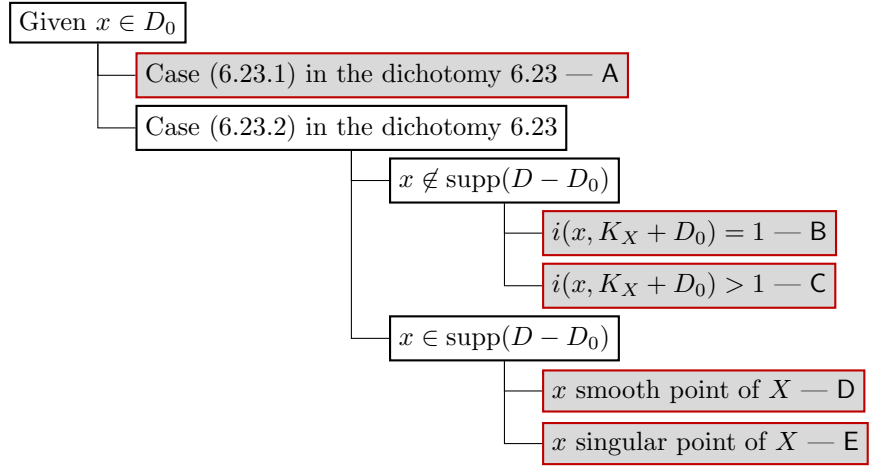
\begin{figure}
  \centering

  \tikzstyle{every node}=[draw=black,thick,anchor=west]
  \tikzstyle{selected}=[draw=red,fill=red!30]
  \tikzstyle{optional}=[dashed,fill=gray!50]
  \begin{tikzpicture}[%
    grow via three points={one child at (0.5,-0.7) and
      two children at (0.5,-0.7) and (0.5,-1.4)},
    edge from parent path={(\tikzparentnode.south) |- (\tikzchildnode.west)}]
    \node {Given $x \in D_0$}
    child { node[draw=red!75!black, fill=black!15] {Case (\ref{awlog:dichotomy}.\ref{itm:dich1}) in the dichotomy \ref{awlog:dichotomy} --- $\sf A$} }
    child { node {Case (\ref{awlog:dichotomy}.\ref{itm:dich2}) in the dichotomy \ref{awlog:dichotomy}}
      child { node {$x \not \in \supp(D-D_0)$}
        child { node[draw=red!75!black, fill=black!15] {$i(x,K_X+D_0) = 1$ --- $\sf B$}}
        child { node[draw=red!75!black, fill=black!15] {$i(x,K_X+D_0) > 1$ --- $\sf C$}}
      }
      child [missing] {}
      child [missing] {}
      child { node {$x \in \supp(D-D_0)$}
        child { node[draw=red!75!black, fill=black!15] {$x$ smooth point of $X$ --- $\sf D$}}
        child { node[draw=red!75!black, fill=black!15] {$x$ singular point of $X$ --- $\sf E$}}
      }
    };
  \end{tikzpicture}
  
  \caption{Cases considered in the proof of Claim~\ref{clm:step 2}}
  \label{fig:claim2}
\end{figure}

We prove Claim \ref{clm:step 2} in the remaining part of Step 2.
The proof is by a distinction of five cases $\sf A$--$\sf E$, which are defined in Figure \ref{fig:claim2}.
Fix an index $\alpha$ such that $x \in U_\alpha$.

\begin{proof}[Proof of Claim~\ref{clm:step 2} in case $\sf A$]
  If we are in case $\sf A$, then the pair $(X, D)$ is snc and reduced at $x$,
  and there are exactly two components of $D$ passing through $x$.
  So we may find a smaller neighborhood $x \in U \subset U_\alpha$ such that
  $D|_U$ has exactly two components. Shrinking $U \ni x$ further, we can also
  assume that $(D - D_0)|_U$ is principal. Now setting $V = U$ and $\gamma = \id\!: V \to U$,
  we are done.
\end{proof}

Note that in all remaining cases $\sf B$--$\sf E$, the pair $(X, D)$ is plt at $x$, and $\rd D.$ is smooth at $x$.

\begin{proof}[Proof of Claim~\ref{clm:step 2} in case $\sf B$]
  In this case, there is a smaller neighborhood $x \in U \subset U_\alpha$ such
  that $D|_U = D_0|_U$. Since $\gamma_\alpha\!: V_\alpha \to U_\alpha$ is \'etale
  over $x$ and $V_\alpha$ is smooth, $x \in U$ is also smooth. Then by Assumption
  \ref{awlog:701}, $(U, D|_U)$ is snc, so we may set $V = U$ and
  $\gamma = \id\!: V \to U$.
\end{proof}

\begin{proof}[Proof of Claim~\ref{clm:step 2} in case $\sf C$]
  Again, there is a neighborhood $x \in U \subset U_\alpha$ such that
  $D|_U = D_0|_U$. We set $V = \gamma_\alpha^{-1}(U)$ and $\gamma =
  \gamma_\alpha|_V$. In this case, we need to show (\ref{clm:step 2}.\ref{itm:step 2.4})
  explicitly. If $\gamma^* D_0$ was reducible, its
  components would meet in the unique point of $\gamma^{-1}(x)$. But $\gamma^*
  D_0$ is normal by \cite[Prop.~5.51]{KM98}, since $(V, \gamma^* D_0)$ is plt.
\end{proof}

In the remaining cases $\sf D$--$\sf E$, we have $x \in \supp (D - D_0)$. Then by Assumption \ref{awlog:690}, there exists a neighborhood $x \in U \subset U_\alpha$ and a component $D_1 \subset D$ such that
  \begin{sequation}\label{eqn:938}
    D|_U = \bigl(D_0 + \textstyle\frac{n_1-1}{n_1} D_1\bigr)\big|_U \text{ for
      some $1 < n_1 < \infty$, and $x \in D_0 \cap D_1$.}
  \end{sequation}%

\begin{proof}[Proof of Claim~\ref{clm:step 2} in case $\sf D$]
  If $x$ is a smooth point of $X$, then the pair $(X, \rp D.)$ is snc by Assumption \ref{awlog:701}. Since $D_1$ is
  Cartier, we may shrink $U \ni x$ and assume that $D_1$ is a principal
  divisor. Then again we may set $U = U_\alpha$ and $\gamma = \id\!: V = U \to U$.
\end{proof}

\begin{proof}[Proof of Claim~\ref{clm:step 2} in case $\sf E$]
  Shrinking $U \ni x$, we may assume
  that the divisor $\gamma_\alpha^* D_1$ is principal on the open set
  $\gamma_\alpha^{-1}(U)$. Then we set $V = \gamma_\alpha^{-1}(U)$ and $\gamma
  = \gamma_\alpha|_V$. It remains to show that the divisors $\gamma_\alpha^*
  D_0$ and $\gamma_\alpha^* D_1$ are both irreducible.
  
  Assuming the contrary, $\gamma_\alpha^* D$ would have at least three distinct
  components, all of which contain $\gamma_\alpha^{-1}(x)$. Since $x \in X_\sg$,
  the fiber $\gamma_\alpha^{-1}(x)$ consists of a single point. So three
  components of $\gamma_\alpha^* D$ would intersect in a point. This, however,
  contradicts Assumption \ref{awlog:723}.
\end{proof}

We have proven Claim~\ref{clm:step 2} in cases $\sf A$--$\sf E$, and these are all the cases. Step~2 is thus finished.

\subsubsection*{Step 3: Construction of $\beta$}

Again, we formulate the statement of Step 3 in a separate claim.

\begin{clm} \label{clm:step 3}
In the setting of Claim \ref{clm:step 2}, there exists a finite morphism $\beta\!: W \to V$ with the following properties.
\begin{enumerate}
\item\label{itm:step 3.1} The map $\beta$ is adapted for $(V, \gamma^* D)$, and the branch locus of $\beta$ is exactly $\supp \{ D_V \}$.
\item\label{itm:step 3.2} The pair $(W, \beta^* \rd \gamma^* D.)$ is reduced and snc.
\item\label{itm:step 3.3} The divisor $E_0 := \beta^{-1}(\gamma^{-1}(D_0))$ is irreducible.
\end{enumerate}
\end{clm}

\begin{proof}
If the divisor $D_V$ defined in Claim \ref{clm:step 2} is zero or reduced, then we may simply set $W = V$ and $\beta = \id\!: W \to V$. Otherwise, we can apply Lemma \ref{lem:ex adapt} to obtain an adapted morphism $\beta\!: W \to V$ for the pair $(V, \gamma^* D)$.
By (\ref{lem:ex adapt}.\ref{itm:ex adapt.1}), it satisfies (\ref{clm:step 3}.\ref{itm:step 3.2}) and (\ref{clm:step 3}.\ref{itm:step 3.3}).
By (\ref{lem:ex adapt}.\ref{itm:ex adapt.2}) and (\ref{lem:ex adapt}.\ref{itm:ex adapt.3}), it satisfies (\ref{clm:step 3}.\ref{itm:step 3.1}).
\end{proof}

\subsubsection*{Step 4: Definition of $\delta$, end of proof}

In order to construct the covering $(U_j)_{j \in J}$ and the maps $\delta_j\!: W_j \to U_j$ whose existence is claimed in Proposition \ref{prp:dlt adapt}, we need to show that for every point $x \in D_0$, there exists a neighborhood $x \in U$ and a map $\delta\!: W \to U$ satisfying (\ref{prp:dlt adapt}.\ref{itm:dlt adapt.1})--(\ref{prp:dlt adapt}.\ref{itm:dlt adapt.3b}).

Given any point $x \in D_0$, let $\gamma\!: V \to U$ be as in Claim~\ref{clm:step 2}, and let $\beta\!: W \to V$ be as in Claim~\ref{clm:step 3}. Set $\delta = \gamma \circ \beta\!: W \to U$. We show that $\delta$ satisfies properties (\ref{prp:dlt adapt}.\ref{itm:dlt adapt.1})--(\ref{prp:dlt adapt}.\ref{itm:dlt adapt.3b}). 
Property (\ref{clm:step 3}.\ref{itm:step 3.2}) translates to (\ref{prp:dlt adapt}.\ref{itm:dlt adapt.1}). Since $\gamma$ is \'etale in codimension one, (\ref{clm:step 3}.\ref{itm:step 3.1}) implies (\ref{prp:dlt adapt}.\ref{itm:dlt adapt.2}). And (\ref{clm:step 3}.\ref{itm:step 3.2}) combined with (\ref{clm:step 3}.\ref{itm:step 3.3}) yields (\ref{prp:dlt adapt}.\ref{itm:dlt adapt.3a}).

We turn to (\ref{prp:dlt adapt}.\ref{itm:dlt adapt.3b}), which says that $\delta|_{E_0}\!: E_0 \to D_0 \cap U$ is an adapted morphism for the pair $(D_0 \cap U, D_0^c|_{D_0 \cap U})$.
This means that if $R \subset D_0 \cap U$ is any irreducible divisor in $\supp \{ D_0^c \}$ of \CC-multiplicity $\ell < \infty$, and $T \subset E_0$ is an irreducible divisor in $E_0$ which surjects onto $R$, then $T$ appears in $(\delta|_{E_0})^* R$ with multiplicity exactly $\ell$.

Let $m$ be the index of $K_X + D_0$ at the general point of $R$. Let $n$ be the \CC-multiplicity of the unique component of $D - D_0$ passing through $R$, or set $n = 1$ if there is no such component. Then by Example \ref{exm:547}, we have $\ell = m \cdot n$. In particular, $n < \infty$.

We will calculate the multiplicity of $T$ in $(\delta|_{E_0})^* R$. We use that $D_0$ is smooth at the general point of $R$, and that $\delta|_{E_0}$ decomposes as
\[ \delta|_{E_0} = \gamma|_{\gamma^{-1}(D_0)} \circ \beta|_{E_0}. \]
\begin{itemize}
\item By (\ref{clm:step 2}.\ref{itm:step 2.2}) and (\ref{clm:step 2}.\ref{itm:step 2.3}), every component of $\bigl(\gamma|_{\gamma^{-1}(D_0)}\bigr)^* R$ has multiplicity $m$.

\item If $n = 1$, then $\beta|_{E_0}$ is unramified at the general point of $T$, because $\beta$ branches only over $D_V$ by (\ref{clm:step 3}.\ref{itm:step 3.1}). 

\item If $1 < n < \infty$, then by (\ref{clm:step 3}.\ref{itm:step 3.1}) we get that $\beta|_{E_0}$ is ramified along $T$ of order $n$.
\end{itemize}
Hence the multiplicity of $T$ in $(\delta|_{E_0})^* R$ is $m \cdot n = \ell$. Claim (\ref{prp:dlt adapt}.\ref{itm:dlt adapt.3b}) is thus shown.
This finishes the proof of Proposition \ref{prp:dlt adapt}. \qed

\subsection{Proof of Theorem \ref{thm:res of symm diff}} \label{subsec:652}

We will only prove the existence of the residue map as asserted in (\ref{thm:res of symm diff}.\ref{itm:6.12.1}). The proof for the restriction map, (\ref{thm:res of symm diff}.\ref{itm:6.12.2}), is analogous. So let $U \subset X$ be an open set, and let $\sigma \in \mathit\Gamma \big( U, \Sym_\CC^{[k]} \Omega_X^p(\log D) \big)$ be any differential form. To simplify notation, we may shrink $X$ and assume that $\sigma$ is globally defined.

Recall from Proposition \ref{prp:dlt adjunction}.\ref{itm:6.8.1+eps} that if $W$ is any component of the non-snc locus of $(X, \rp D.)$ such that $W \cap D_0$ is a divisor in $D_0$, then $W = W \cap D_0$ and $W$ appears in $\supp \rp D_0^c.$. Hence considering the $k$-th symmetric power of the residue map on the snc locus of $(X, \rp D.)$, we obtain a section
\[ \widetilde\sigma \in H^0 \bigl( D_0, \bigl( \Sym^{[k]} \Omega_{D_0}^{p-1} \bigr) (*\rp D_0^c.) \bigr) \]
with possibly arbitrarily high order poles along $\supp \rp D_0^c.$, cf.~Definition \ref{dfn:383}. We need to show that in fact $\widetilde\sigma$ only has orbifold poles, that is
\begin{sequation} \label{seq:562}
\widetilde\sigma \in H^0 \bigl( D_0, \Sym_\CC^{[k]} \Omega_{D_0}^{p-1}(\log D_0^c) \bigr) \subset H^0 \bigl( D_0, \bigl( \Sym^{[k]} \Omega_{D_0}^{p-1} \bigr) (*\rp D_0^c.) \bigr).
\end{sequation}%
We can then define $\res_{D_0}^k(\sigma) := \widetilde\sigma$. For this, we employ the criterion of \cite[Cor.~3.16]{JK11}: Inclusion \eqref{seq:562} holds if and only if the pullback of $\widetilde\sigma$ via an adapted morphism is a regular logarithmic form.

Since the sheaf $\Sym_\CC^{[k]} \Omega_{D_0}^{p-1}(\log D_0^c)$ is reflexive, we may remove a small subset from $D_0$ and assume that $D_0$ is covered by maps $\delta_j\!: W_j \to U_j$ as in Proposition \ref{prp:dlt adapt}. For any index $j \in J$, write $E_j = \delta_j^* \rd D.$, and set $E_{0,j} = \delta_j^{-1}(D_0)$. By (\ref{prp:dlt adapt}.\ref{itm:dlt adapt.3b}), $\delta_j|_{E_{0,j}}$ is an adapted morphism for the pair $(D_0 \cap U_j, D_0^c|_{D_0 \cap U_j})$. Hence by \cite[Cor.~3.16]{JK11}, to prove \eqref{seq:562} it suffices to show that
\begin{sequation} \label{seq:568}
(\delta|_{E_0})^* ( \widetilde\sigma ) \in H^0 \Bigl( E_0, \Sym^{[k]} \Omega_{E_0}^{p-1} \bigl( \log (E - E_0)\big|_{E_0} \bigr) \Bigr)
\end{sequation}%
for every $j \in J$. To see \eqref{seq:568}, note that by (\ref{prp:dlt adapt}.\ref{itm:dlt adapt.1}), the pair $(W_j, E_j)$ is snc, hence there is a residue map
\[ \begin{array}{rccl}
\res_{E_0}^k\!: & \Sym^{[k]} \Omega_{W_j}^p(\log E_j) & =   & \Sym^k \Omega_{W_j}^p(\log E_j) \\
                &                                     & \to & \Sym^k \Omega_{E_{0,j}}^{p-1}\bigl(\log (E_j - E_{0,j})\big|_{E_{0,j}}\bigr)      \\
                &                                     & =   & \Sym^{[k]} \Omega_{E_{0,j}}^{p-1}\bigl(\log (E_j - E_{0,j})\big|_{E_{0,j}}\bigr).
\end{array} \]
And by (\ref{prp:dlt adapt}.\ref{itm:dlt adapt.2}), $\delta_j$ is adapted for the pair $(U_j, D|_{U_j})$, so by the definition of \CC-differentials, we have
\[ \delta^* \sigma \in H^0 \bigl( W_j, \Sym^{[k]} \Omega_{W_j}^p(\log E_j) \bigr). \]
Recall that the standard residue map \eqref{seq:snc res map} commutes with \'etale pullback, and that $\delta_j$ is \'etale over the general point of $D_{0,j}$ by (\ref{prp:dlt adapt}.\ref{itm:dlt adapt.1}). Hence we see that the two forms
\[ (\delta_j|_{E_{0,j}})^* ( \widetilde\sigma ) \in H^0 \bigl( E_{0,j}, \Sym^{[k]} \Omega_{E_{0,j}}^{p-1}\bigl(\log (E_j - E_{0,j})\big|_{E_{0,j}}\bigr)(*\supp \delta_j|_{E_{0,j}}^* \rp D_0^c.) \bigr) \]
and
\[ \res_{E_{0,j}}^k(\delta_j^* \sigma) \in H^0 \bigl( E_{0,j}, \Sym^{[k]} \Omega_{E_{0,j}}^{p-1}\bigl(\log (E_j - E_{0,j})\big|_{E_{0,j}}\bigr) \bigr) \]
agree on an open set of $E_{0,j}$. This shows inclusion \eqref{seq:568} and ends the proof. \qed

\section{\texorpdfstring{Relative differentials on \CC-pairs}{Relative differentials on C-pairs}} \label{sec:rel diff C-pairs}

\subsection{Relative differentials in the smooth setting} \label{subsec:rel diff sm}

Consider a smooth map of smooth varieties $f\!: X \to Y$, with general fiber $F$. The relative differential sequence for $f$ \cite[Ch.~II, Prop.~8.11]{Har77} reads
\begin{sequation} \label{eqn:rel diff}
0 \to f^* \Omega_Y^1 \to \Omega_X^1 \to \Omega_{X/Y}^1 \to 0.
\end{sequation}%
For $\Omega_X^p$ with $p > 1$, there is no analogous sequence, but \eqref{eqn:rel diff} induces a filtration
\[ \Omega_X^p = \sF^0 \supset \sF^1 \supset \cdots \supset \sF^p \supset \sF^{p+1} = 0 \]
with successive quotients
\[ \sF^i / \sF^{i+1} \isom f^* \Omega_Y^i \tensor \Omega_{X/Y}^{p-i}. \]
Now let $\sA \subset \Omega_X^p$ be an invertible sheaf of $p$-forms on $X$, and let $j$ be the largest number such that $\sA \subset \sF^j$. Then $\sA$ injects into $f^* \Omega_Y^j \tensor \Omega_{X/Y}^q$, where $q = p - j$. Restricting to the general fiber $F$ and choosing a trivialization of $f^* \Omega_Y^j\big|_F$, we obtain embeddings
\begin{sequation} \label{eqn:emb}
\sA|_F \inj \Omega_F^q.
\end{sequation}%
These embeddings are not canonical, but the number $q$ is.

\subsection{Main result of this section}

We will generalize the construction described above in three different directions. Firstly, we allow singular varieties, and accordingly we consider Weil divisorial rather than invertible sheaves. Secondly, we also deal with symmetric differential forms. Lastly, everything is formulated in the setting of \CC-pairs.

\begin{prp} \label{prp:rel diff C-pairs}
Let $(X, D)$ be a \CC-pair, with $X$ irreducible, and let $f\!: X \to Y$ be a surjective morphism to a variety $Y$, with general fiber $F$. Then $F$ is normal and $(F, D|_F)$ is a \CC-pair.
Let $\sA \subset \Sym_\CC^{[1]} \Omega_X^p(\log D)$ be a Weil divisorial subsheaf. Then there exists a number $0 \le q \le p$ and a sequence of embeddings
\[ \alpha_k\!: \bigl( \Sym_\CC^{[k]} \sA \bigr)\big|_F^{**} \inj \Sym_\CC^{[k]} \Omega_F^q(\log D|_F) \]
for all $k \in \N$. Setting
\[ X^\circ := X_\sm \minus \supp \rp D. \quad \text{and} \quad F^\circ := F \cap X^\circ, \]
noting that
\[ \bigl( \Sym_\CC^{[k]} \sA \bigr)\big|_{F^\circ}^{**} = \Sym^k (\sA|_{F^\circ}) \quad \text{and} \quad \Sym_\CC^{[k]} \Omega_{F^\circ}^q(\log D|_{F^\circ}) = \Sym^k \Omega_{F^\circ}^q, \]
and that $\sA|_{X^\circ}$ is invertible, the following compatibility properties (\ref{prp:rel diff C-pairs}.\ref{itm:cp1}--\ref{itm:cp2}) hold.
\begin{enumerate}
\item\label{itm:cp1} The map $\alpha_1|_{F^\circ}$ is one of the maps constructed in Section \ref{subsec:rel diff sm}.

\item\label{itm:cp2} For any $k$, the sheaves $\img(\alpha_k)$ and $(\img \alpha_1)^{[k]}$ generically agree as subsheaves of $\Sym_\CC^{[k]} \Omega_F^q(\log D|_F)$, i.e.~we have a commutative diagram
\[ \xymatrix{
\Sym^k (\sA|_{F^\circ}) \ar@{=}[r] \ar@{=}[d] & \Sym^k (\sA|_{F^\circ}) \ar@{=}[d] \\
\Sym^k \bigl( (\Sym_\CC^{[1]} \sA )\big|_{F^\circ}^{**} \bigr) \ar[d]_-{\Sym^k \alpha_1} &
  \bigl( \Sym_\CC^{[k]} \sA \bigr)\big|_{F^\circ}^{**} \ar[d]^-{\alpha_k} \\
\Sym^k \bigl( \Sym_\CC^{[1]} \Omega_{F^\circ}^q(\log D|_{F^\circ}) \bigr) \ar@{=}[d] &
  \Sym_\CC^{[k]} \Omega_{F^\circ}^q(\log D|_{F^\circ}) \ar@{=}[d] \\
\Sym^k \Omega_{F^\circ}^q \ar@{=}[r] & \Sym^k \Omega_{F^\circ}^q.
} \]
\end{enumerate}
\end{prp}

\begin{proof}
The proof is divided into four steps.

\subsubsection*{Simplifying assumptions}
By Lemma \ref{lem:ext morphisms}, it suffices to construct the maps $\alpha_k$ outside a small subset of $F$. An easy dimension count shows that if $Z \subset X$ is any small subset, then $Z \cap F$ is a small subset of $F$. Hence we may remove the non-snc locus of $(X, \rp D.)$ from $X$, as well as the intersection of any two distinct components of $\supp D$. Furthermore, we may apply generic smoothness \cite[Rem.~2.11]{GKKP11} to $f$. Put together, we are free to make the following simplifying assumptions.
\begin{awlog} \label{awlog:768}
The pair $(X, \rp D.)$ is snc, $\supp D$ is locally irreducible, $Y$ is smooth, and $\rp D.$ is relatively snc over $Y$.
\end{awlog}
In particular, all the sheaves occurring will be locally free, so we do not need to take double duals, allowing us to drop the usual square brackets $^{[\cdot]}$.

\subsubsection*{Construction of an embedding as in Section \ref{subsec:rel diff sm}}

Let $\sB$ be the saturation of $\sA$ in $\Omega_X^p(\log \rp D.)$. Consider the relative log differential sequence \cite[(10.1)]{GKKP11}
\[ 0 \to f^* \Omega_Y^1 \to \Omega_X^1 (\log \rp D.) \to \Omega_{X/Y}^1 (\log \rp D.) \to 0. \]
By Assumption \ref{awlog:768}, all the sheaves in this sequence are locally free. Hence by Lemma \ref{lem:reduction}, $\sB$ injects into $f^* \Omega_Y^i \tensor \Omega_{X/Y}^{p-i} (\log \rp D.)$ for some $0 \le i \le p$. Restricting to $F$, we see that
\[ \sB|_F \inj \bigl( \Omega_F^q (\log \rp D|_F.) \bigr)^{\oplus \rk \Omega_Y^i} \]
for $q = p - i$. Note that this $q$ is the same number as in Section \ref{subsec:rel diff sm}. Choose a summand of the sheaf on the right-hand side such that the projection to this summand, restricted to $\sB|_F$, is nonzero. This way,  we obtain an embedding
\[ \beta\!: \sB|_F \inj \Omega_F^q (\log \rp D|_F.). \]

\subsubsection*{Construction of the maps $\alpha_k$}
Now let $k$ be any natural number. By Lemma \ref{lem:gen eq}, we have an inclusion
\[ \Sym_\CC^k \sA \subset \sB^k := \sB^{\tensor k} \subset \Sym^k \Omega_X^p(\log \rp D.). \]
Taking symmetric powers of $\beta$, we also get embeddings
\[ \beta_k\!: \underbrace{(\sB^k)|_F}_{= (\sB|_F)^k} \inj \Sym^k \Omega_F^q (\log \rp D|_F.). \]
We will show that the restriction of $\beta_k$ to the subsheaf $(\Sym_\CC^k \,\sA)|_F \subset (\sB^k)|_F$ factors via $\Sym_\CC^k \,\Omega_F^q(\log D|_F)$, as indicated in the following diagram.
\begin{sequation} \label{seq:607}
\xymatrix{
(\sB^k)|_F \ar@{ ir->}[rr]^-{\beta_k}                        & & \Sym^k \Omega_F^q (\log \rp D|_F.)                    \\
(\Sym_\CC^k \,\sA)|_F \ar@{-->}[rr]^-{\alpha_k} \ar@{ ir->}[u] & & \Sym_\CC^k \,\Omega_F^q(\log D|_F) \ar@{ ir->}[u]     \\
}
\end{sequation}%
It is then clear that the maps $\alpha_k$ defined by the above diagram are embeddings and satisfy the compatibility conditions (\ref{prp:rel diff C-pairs}.\ref{itm:cp1})--(\ref{prp:rel diff C-pairs}.\ref{itm:cp2}).

\subsubsection*{Calculation in local coordinates}
The existence of diagram \eqref{seq:607} is a statement about pole orders, which can be checked by an elementary calculation in local analytic coordinates.\
\PreprintAndPublication{
So from now on, we will assume the following, additionally to Assumption \ref{awlog:768}:
\begin{awlog} \label{awlog:793}
The following hold.
\begin{enumerate}
\item $F$ is open in $\C^n$ with coordinates $z_1, \dots, z_n$.
\item $Y$ is open in $\C^m$ with coordinates $y_1, \dots, y_m$.
\item $X = Y \x F$ is open in $\C^{n+m}$ with coordinates $x_1, \dots, x_{n+m}$.
\item The map $f\!: X \to Y$ is given by projection to the first factor.
\item $\mathrm{Supp} \, D = \{ x_1 = 0 \}$.
\item The sheaf $\sB$ is isomorphic to $\O_X$.
\item \label{itm:803} $D$ is non-reduced.
\end{enumerate}
\end{awlog}
Assumption (\ref{awlog:793}.\ref{itm:803}) is justified by the observation that if $D$ is reduced, we can simply set $\alpha_k = \beta_k$.

First we introduce some notation. By an \emph{$r$-$\ell$-multi-index} we mean an element $I = (i_1, \dots, i_r) \in \{ 1, \dots, \ell \}^r$ with $i_1 < \dots < i_r$. The set $\cI_r^\ell$ of $r$-$\ell$-multi-indices is endowed with the lexicographical ordering ``$\le$''.
Set $\delta x_1 = \dif x_1/x_1$ and $\delta x_i = \dif x_i$ for $i \ge 2$. Analogously for $\delta z_i$.
For any multi-index $I \in \cI_r^{n+m}$, we define $\dif x_I = \dif x_{i_1} \wedge \dots \wedge \dif x_{i_r}$ and $\delta x_I = \delta x_{i_1} \wedge \dots \wedge \delta x_{i_r}$. Analogously for $\delta z_I$, where $I \in \cI_r^n$.
Thus $\Omega_X^r(\log \rp D.)$ is freely generated by $\{ \delta x_I \,|\, I \in \cI_r^{n+m} \}$, and $\Omega_F^r(\log \rp D.)$ is freely generated by $\{ \delta z_I \,|\, I \in \cI_r^n \}$.

Let $\sigma_0 = \sum_{I \in \cI_p^{n+m}} b_I \delta x_I$ be a generator of $\sB$, and let $J$ be a multi-index such that $b_J$ is not constantly zero. We may uniquely write $J = (J_1, J_2)$ where $J_1 \in \cI_q^n$ and all the entries of $J_2$ are $\ge n + 1$. The map $\beta\!: \sB|_F \inj \Omega_F^q(\log \rp D|_F.)$ is then given by
\[ \sigma = \sum_{I \in \cI_p^{n+m}} a_I \delta x_I \mapsto \sum_{I \in \cI_q^n} a_{(I, J_2)}|_F \,\delta z_I. \]
The induced maps $\beta_k\!: (\sB^k)|_F \inj \Sym^k \Omega_F^q(\log \rp D|_F.)$ are defined by the property that $\beta_k(\sigma_0^k) = \beta(\sigma_0)^k$, which in local coordinates reads
\begin{sequation} \label{eqn:beta_k}
\sigma = \sum_{\substack{I_1 \le \cdots \le I_k \\[.5ex] I_j \in \cI_p^{n+m}}} a_{I_1, \dots, I_k} \delta x_{I_1} \cdots \delta x_{I_k} \mapsto \sum_{\substack{I_1 \le \cdots \le I_k \\[.5ex] I_j \in \cI_q^n}} a_{(I_1, J_2), \dots, (I_k, J_2)}|_F \,\delta z_{I_1} \cdots \delta z_{I_k}.
\end{sequation}%
Let $n_1 < \infty$ be the \CC-multiplicity of $D$. A calculation in the spirit of \cite[Sec.~3.C]{JK11} shows that $\Sym_\CC^k \,\Omega_X^p(\log D)$ is freely generated by
\begin{sequation} \label{eqn:generators X}
\left\{ x_1^{- \rd m_1 \cdot (1 - \frac 1{n_1}).} \dif x_{I_1} \cdots \dif x_{I_k} \;\middle|\; I_1 \le \cdots \le I_k \in \cI_p^{n+m} \right\},
\end{sequation}%
where $m_1 = \#\{ \lambda \;|\; I_\lambda(1) = 1 \}$.
Similarly, $\Sym_\CC^k \,\Omega_F^q(\log D|_F)$ is freely generated by
\begin{sequation} \label{eqn:generators F}
\left\{ z_1^{- \rd m_1 \cdot (1 - \frac 1{n_1}).} \dif z_{I_1} \cdots \dif z_{I_k} \;\middle|\; I_1 \le \cdots \le I_k \in \cI_q^n \right\}.
\end{sequation}%
From \eqref{eqn:beta_k}--\eqref{eqn:generators F}, it is immediate that if $\sigma$ is a local section of $\Sym_\CC^k \,\sA \subset \sB^k \cap \Sym_\CC^k \,\Omega_X^p(\log D)$, then $\beta_k(\sigma) \in \Sym_\CC^k \,\Omega_F^q(\log D|_F)$.
}{
For the details, we refer to the preprint version of this paper, which may be found on the arXiv.
}
\end{proof}

\section{Minimal dlt models} \label{sec:min dlt models}

In this section, we consider certain partial resolutions of log canonical pairs, called \emph{minimal dlt models}. We will use these in the proof of Theorem \ref{thm:lc BS van}.

\begin{dfn}[Minimal dlt model] \label{dfn:mindltmodel}
Let $(X, D)$ be a log canonical pair, and let $f\!: Z \to X$ be a proper birational morphism from a normal variety $Z$. Let $E$ denote the divisorial part of $\Exc(f)$ with its reduced structure. We say that $f\!: Z \to X$ is a \emph{minimal dlt model} of $(X, D)$ if the following properties hold.
\begin{enumerate}
\item\label{itm:mindlt1} The variety $Z$ is $\Q$-factorial.
\item\label{itm:mindlt2} The pair $(Z, f^{-1}_{\;\;\;*} D + E)$ is dlt.
\item\label{itm:mindlt3} Over the snc locus of $(X, \rp D.)$, the map $f$ is an isomorphism.
\item\label{itm:mindlt4} The map $f$ only extracts divisors of discrepancy $-1$, i.e.~the ramification formula reads
\[ K_Z + f^{-1}_{\;\;\;*} D = f^*(K_X + D) - E. \]
\end{enumerate}
\end{dfn}

The first thing we need to know about minimal dlt models is that they exist. This is due to Hacon and ultimately to \cite{BCHM10}. For a reference, see \cite[Thm.~3.1]{KK10}.

\begin{thm}[Existence of minimal dlt models] \label{thm:ex mindltmodel}
Let $(X, D)$ be a log canonical pair. Then $(X, D)$ has a minimal dlt model. \qed
\end{thm}

We will state and prove the properties of minimal dlt models $f\!: Z \to X$ which are relevant to us in two lemmas. The first one concerns the absolute properties of $Z$, while the latter deals with the fibers of $f$.

\begin{lem}[Minimal dlt models I] \label{lem:mindltmodel 1}
Let $(X, D')$ be a log canonical pair, and let $f\!: Z \to X$ be a minimal dlt model of $(X, D')$.
Denote by $E_1, \dots, E_k$ the codimension-one components of $\Exc(f)$, and let $E = E_1 + \cdots + E_k$. Let $D \le D'$ be a divisor such that $(X, D)$ is a \CC-pair. Set $D_Z := f^{-1}_{\;\;\;*} D + E$. Then
\begin{enumerate}
\item\label{itm:mindlt1.1} the pair $(Z, D_Z)$ is a dlt \CC-pair, and
\item\label{itm:mindlt1.2} $K_Z + D_Z \sim_{\R, X} - f^{-1}_{\;\;\;*} (D' - D)$.
\end{enumerate}
For any $1 \le i \le k$, let $E_i^c := \Diff_{E_i}(D_Z - E_i)$ denote the different, as in Proposition/Definition \ref{prpdfn:different}.
Then the following properties hold.
\begin{enumerate}
\setcounter{enumi}{2}
\item\label{itm:mindlt1.3} The pair $(E_i, E_i^c)$ is a dlt \CC-pair.
\item\label{itm:mindlt1.4} $K_{E_i} + E_i^c$ is $\R$-linearly equivalent over $f(E_i)$ to an anti-effective divisor.
\end{enumerate}
\end{lem}

\begin{proof}
By the definition of minimal dlt model, $(Z, f^{-1}_{\;\;\;*} D' + E)$ is dlt, and
\[ K_Z + f^{-1}_{\;\;\;*} D' + E = f^*(K_X + D') \sim_{\R, X} 0. \]
Since $Z$ is $\Q$-factorial, (\ref{lem:mindltmodel 1}.\ref{itm:mindlt1.1}) follows from the monotonicity of discrepancies \cite[Lem.~2.27]{KM98}. Claim (\ref{lem:mindltmodel 1}.\ref{itm:mindlt1.2}) is clear.

Claim (\ref{lem:mindltmodel 1}.\ref{itm:mindlt1.3}) is an application of Proposition \ref{prp:dlt adjunction}. Recall from Equation \eqref{eqn:different} on Page \pageref{eqn:different} that
\[ K_{E_i} + E_i^c \sim_\Q (K_Z + D_Z)|_{E_i}. \]
Hence claim (\ref{lem:mindltmodel 1}.\ref{itm:mindlt1.4}) follows from (\ref{lem:mindltmodel 1}.\ref{itm:mindlt1.2}) upon noting that the effective divisor $f^{-1}_{\;\;\;*} (D' - D)$ does not contain any of the $E_i$'s in its support.
\end{proof}

\begin{lem}[Minimal dlt models II] \label{lem:mindltmodel 3}
In the setup of Lemma \ref{lem:mindltmodel 1}, let $1 \le i \le k$ be any index and denote by $F_i$ a general fiber of $f|_{E_i}\!: E_i \to f(E_i)$. Then:
\begin{enumerate}
\item\label{itm:mindlt3.2} $F_i$ is normal (but not necessarily connected) and the pair $(F_i, E_i^c|_{F_i})$ is a dlt \CC-pair.
\item\label{itm:mindlt3.3} $K_{F_i} + E_i^c|_{F_i}$ is $\R$-linearly equivalent to an anti-effective divisor.
\end{enumerate}
\end{lem}

\begin{proof}
Claim (\ref{lem:mindltmodel 3}.\ref{itm:mindlt3.2}) holds by \cite[Lemma 2.25]{GKKP11} applied to the pair $(E_i, E_i^c)$.
Since $E_i$ is normal, an easy dimension count shows that $\codim_{F_i} (E_i^\sg \cap F_i) \ge 2$. Hence to check (\ref{lem:mindltmodel 3}.\ref{itm:mindlt3.3}), it suffices to do so on $F_i \cap E_i^\sm$. But since the normal bundle of $F_i \cap E_i^\sm$ in $E_i^\sm$ is trivial, we have $(K_{E_i} + E_i^c)|_{F_i} = K_{F_i} + E_i^c|_{F_i}$. So the assertion follows from (\ref{lem:mindltmodel 1}.\ref{itm:mindlt1.4}).
\end{proof}

\section{\texorpdfstring{Weak Bogomolov--Sommese vanishing on dlt \CC-pairs}{Weak Bogomolov--Sommese vanishing on dlt C-pairs}} \label{sec:dlt C-pairs}

The most general known \CC-pair version of Bogomolov--Sommese vanishing, which is \cite[Thm. 7.3]{GKKP11}, works for \Q-factorial log canonical pairs. Here we would like to make the point that if we restrict ourselves to dlt pairs, we can easily dispose of the \Q-factoriality assumption. Due to the existence of \emph{$\Q$-factorializations}, this is much easier than our main result, yet we will use it in the proof of the latter.

\begin{dfn}[$\Q$-factorialization] \label{dfn:Q-fact}
Let $(X, D)$ be a dlt pair. A \emph{$\Q$-factorial\-ization} of $(X, D)$ is a proper birational morphism $f\!: Z \to X$ from a normal variety $Z$ such that the following holds.
\begin{enumerate}
\item\label{itm:Qfact1} The variety $Z$ is $\Q$-factorial.
\item\label{itm:Qfact2} The $f$-exceptional locus is a small subset of $Z$.
\item\label{itm:Qfact3} The pair $(Z, f^{-1}_{\;\;\;*} D)$ is dlt.
\end{enumerate}
\end{dfn}

\Q-factorializations are a special case of minimal dlt models. Since these exist
by Theorem \ref{thm:ex mindltmodel}, we get the following existence statement.

\begin{thm}[Existence of $\Q$-factorializations] \label{thm:ex Q-fact}
Let $(X, D)$ be a dlt pair. Then $(X, D)$ admits a $\Q$-factorialization.
\end{thm}

\begin{proof}
By Theorem \ref{thm:ex mindltmodel}, $(X, D)$ has a minimal dlt model $f\!: Z \to X$. We will show that $f$ is a $\Q$-factorialization of $(X, D)$. Property (\ref{dfn:mindltmodel}.\ref{itm:mindlt1}) directly translates to (\ref{dfn:Q-fact}.\ref{itm:Qfact1}). The definition of a dlt pair \cite[Def.~2.37]{KM98} and (\ref{dfn:mindltmodel}.\ref{itm:mindlt3}) imply that $f$ only extracts divisors of discrepancy strictly greater than $-1$. On the other hand, by (\ref{dfn:mindltmodel}.\ref{itm:mindlt4}), $f$ only extracts divisors whose discrepancy is exactly $-1$. So $f$ does not extract any divisors at all. This is (\ref{dfn:Q-fact}.\ref{itm:Qfact2}). Finally, (\ref{dfn:Q-fact}.\ref{itm:Qfact3}) follows immediately from (\ref{dfn:mindltmodel}.\ref{itm:mindlt2}).
\end{proof}

\begin{lem}[Invariance of $\kappa_\CC$ under small morphisms] \label{lem:kappa Q-fact}
Let $(X, D)$ be a \CC-pair, with a Weil divisorial subsheaf $\ms A \subset \Sym_\CC^{[1]} \Omega_X^p(\log D)$. Let $f\!: Z \to X$ be a proper birational morphism from a normal variety $Z$, whose exceptional locus is small in~$Z$. Then $(Z, f^{-1}_{\;\;\;*} D)$ is also a \CC-pair, and there is a natural embedding of $f^{[*]} \ms A$ as a subsheaf of $\Sym_\CC^{[1]} \Omega_Z^p(\log f^{-1}_{\;\;\;*} D)$. With respect to this embedding, we have $\kappa_\CC(f^{[*]} \ms A) = \kappa_\CC(\ms A)$.
\end{lem}

\begin{proof}
The natural embedding $f^{[*]} \ms A \inj \Sym_\CC^{[1]} \Omega_Z^p(\log f^{-1}_{\;\;\;*} D)$ is given by the pullback of differential forms along $f$. The sheaves $f^{[*]} ( \Sym_\CC^{[m]} \ms A )$ and $\Sym_\CC^{[m]} ( f^{[*]} \ms A )$ agree outside the exceptional locus of $f$. Since they are both reflexive, they are in fact isomorphic by Lemma \ref{lem:ext morphisms}(\ref{itm:extmor2}). So
\[ H^0 \bigl( X, \Sym_\CC^{[m]} \ms A \bigr) = H^0 \bigl( Z, f^{[*]} ( \Sym_\CC^{[m]} \ms A ) \bigr) = H^0 \bigl( Z, \Sym_\CC^{[m]} ( f^{[*]} \ms A ) \bigr). \]
The lemma follows, because $\kappa_\CC(\ms A)$ and $\kappa_\CC(f^{[*]} \ms A)$ are determined by the spaces of global sections appearing on the left- and right-hand side of the above equation, respectively.
\end{proof}

\begin{prp}[Weak Bogomolov--Sommese vanishing on dlt \CC-pairs not of log general type] \label{prp:BS van on dlt C-pairs}
Let $(X, D)$ be a projective dlt \CC-pair of dimension $n$ such that $K_X + D$ is not big. If $\sA \subset \Sym_\CC^{[1]} \Omega_X^p(\log D)$ is a Weil divisorial subsheaf for some $p$, then $\kappa_\CC(\sA) \le n - 1$.
\end{prp}

\begin{proof}
By Theorem \ref{thm:ex Q-fact} and Lemma \ref{lem:kappa Q-fact}, we may pass to a $\Q$-factorialization of $(X, D)$ and assume that $\sA$ is $\Q$-Cartier. If $p \le n - 1$, the claim was proven in \cite[Theorem 7.3]{GKKP11}\footnote{Note that the cited theorem erroneously contains the extra assumption that $\dim X \le 3$. However, the proof works in all dimensions.}. Thus we only need to deal with the case $p = n$.

We argue by contradiction and assume that there is $\sA \subset \Sym_\CC^{[1]} \Omega_X^n(\log D)$ with $\kappa_\CC(\sA) = n$. By Equation (\ref{eqn:souffl'e}), this means that $mK_X + \rd mD.$ is big for some $m$. But then also $K_X + D$ is big, contrary to our assumption.
\end{proof}

\part{PROOF AND SHARPNESS OF MAIN RESULT} \label{part:III}

\section{Proof of Theorem \ref*{thm:lc BS van}} \label{sec:proof of main thm}

\PreprintAndPublication{
\begin{figure}[t]
  \centering
  \[
  \xymatrix{
    *++{\begin{tikzpicture}(4,4)(0,0)
        \fill (4.75,1.8) node[right]{\scriptsize $Z$};
        \draw (0,0) to [out=90, in=180] (2,2); 
        \draw (2,2) to [out=0, in=90] (6,0);
        \draw (6,0) to [out=-90, in=-30] (3,-1);
        \draw (3,-1) to [out=150, in=-90] (0,0);
        \fill (5,.5) node[right]{\scriptsize $E_i$};
        \draw (1,-.3) to [out=20, in=180] (2,-.05); 
        \draw (2,-.05) to [out=0, in=185] (5,-.7);
        \draw (1,1.2) to [out=20, in=180] (2,1.45); 
        \draw (2,1.45) to [out=0, in=185] (5,.8);
        \draw (1,-.3) to [out=90, in=-90] (1,1.2); 
        \draw (5,-.7) to [out=90, in=-90] (5,.8); 
        \fill (3.41,.65) node[right]{\scriptsize $F_i$};
        \draw[semithick] (3.5,-.41) to [out=90, in=-90] (3.5,1.09); 
        \draw[ultra thin] (1.5,-.12) to [out=90, in=-90] (1.5,1.38); 
        \draw[ultra thin] (2,-.05) to [out=90, in=-90] (2,1.45);
        \draw[ultra thin] (2.5,-.09) to [out=90, in=-90] (2.5,1.41);
        \draw[ultra thin] (3,-.22) to [out=90, in=-90] (3,1.28);
        \draw[ultra thin] (4,-.59) to [out=90, in=-90] (4,.91);
        \draw[ultra thin] (4.5,-.7) to [out=90, in=-90] (4.5,.8);
    \end{tikzpicture}
    }
    \ar[d]^-f \\
    *++{\begin{tikzpicture}(4,4)(0,0)
        \fill (5.3,1.1) node[right]{\scriptsize $X$};
        \draw (0,0) to [out=90, in=170] (2,1); 
        \draw (2,1) to [out=-10, in=90] (6,0);
        \draw (6,0) to [out=-90, in=-30] (3,-1);
        \draw (3,-1) to [out=150, in=-90] (0,0);
        \fill (4.8,-.3) node[right]{\scriptsize $f(E_i)$};
        \draw (1,-.1) to [out=20, in=180] (2,.15); 
        \draw (2,.15) to [out=0, in=200] (4.8,-.3);
    \end{tikzpicture}
    }
  }
  \]
  \caption{Setup of the proof of Theorem \ref{thm:lc BS van}} \label{fig:proof main thm}
\end{figure}
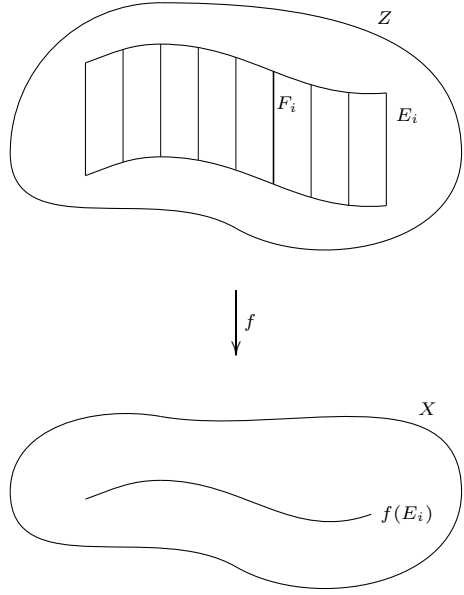
}{}

This section is devoted to the proof of Theorem \ref{thm:lc BS van}. For reasons of clarity, the proof is divided into three separate steps.
\PreprintAndPublication{
Figure \ref{fig:proof main thm} summarizes the notation used.
}{}

\subsection{Step 1: Passing to a minimal dlt model}

Let $f\!: Z \to X$ be a minimal dlt model of $(X, D')$, with $E = E_1 + \cdots + E_\ell$ the divisorial part of $\Exc(f)$ and $D_Z := f^{-1}_{\;\;\;*} D + E$. Clearly $(Z, D_Z)$ is a \CC-pair, and by \cite[Theorem 4.3]{GKKP11}, there exists an embedding
\[ f^{[*]} \ms A \inj \Omega_Z^{[p]}(\log \rd D_Z.) = \Sym_\CC^{[1]} \Omega_Z^p(\log D_Z). \]
Define $\sB$ as the saturation of $f^{[*]} \ms A$ in $\Omega_Z^{[p]}(\log \rd D_Z.)$. By Lemma \ref{lem:saturation reflexive}, $\sB$ is reflexive. Then it is even \Q-Cartier, since $Z$ is \Q-factorial. To prove Theorem \ref{thm:lc BS van}, it is sufficient to show the following claim.
\begin{clm} \label{clm:1060}
We have $\kappa_\CC(\ms B) = \kappa_\CC(\ms A)$.
\end{clm}
For then $\kappa_\CC(\ms A) \le p$ by Bogomolov--Sommese vanishing for the sheaf $\ms B$ on the \Q-factorial dlt \CC-pair $(Z, D_Z)$, as stated in \cite[Thm.~7.3]{GKKP11}. We will prove Claim \ref{clm:1060} in Step 3, after discussing the properties of $\sB$ in Step 2. But before, we make two observations which simplify notation. First, if $E = 0$, then Claim \ref{clm:1060} directly follows from Lemma \ref{lem:kappa Q-fact}. Second, if $p = 0$, then the assertion of Theorem \ref{thm:lc BS van} itself is trivial, because $\Sym_\CC^{[k]} \Omega_X^0(\log D) = \O_X$ for all $k$. Hence, we will assume the following for the rest of the proof.
\begin{awlog}
The divisor $E$ is nonzero, and $p > 0$.
\end{awlog}

\subsection{Step 2: Properties of the sheaf $\sB$}

Let $1 \le i \le \ell$ be any number. We set up some notation.
\begin{ntn}
Let $F_i$ be a general fiber of $f|_{E_i}$, and set $\ms B_i = \big(\ms B|_{E_i}^{**}\big)\big|_{F_i}^{**}$. Furthermore, denote by $E_i^c$ the different $\Diff_{E_i}(D_Z - E_i)$.
\end{ntn}
By Lemma \ref{lem:mindltmodel 3}, the pair $(F_i, E_i^c|_{F_i})$ is a dlt \CC-pair.

\begin{clm}[Embeddings and non-bigness] \label{clm:notbig}
There is a (canonical) embedding
\begin{sequation} \label{eqn:can emb}
\ms B|_{E_i}^{**} \inj \Sym_\CC^{[1]} \Omega_{E_i}^r(\log E_i^c)
\end{sequation}%
for either $r = p - 1$ or $r = p$. It induces (non-canonical) embeddings
\begin{sequation} \label{eqn:alpha_k}
\alpha_k\!: \bigl( \Sym_\CC^{[k]} (\sB|_{E_i}^{**}) \bigr)\big|_{F_i}^{**} \inj \Sym_\CC^{[k]} \Omega_{F_i}^q(\log E_i^c|_{F_i})
\end{sequation}%
for some number $0 \le q \le r$ and all $k \ge 0$. The $\alpha_k$'s satisfy the compatibility condition (\ref{prp:rel diff C-pairs}.\ref{itm:cp2}).
Since $\sB_i \subset \bigl( \Sym_\CC^{[1]} (\sB|_{E_i}^{**}) \bigr)\big|_{F_i}^{**}$, from $\alpha_1$ we obtain an embedding
\begin{sequation} \label{eqn:alpha_1'}
\alpha_1' = \alpha_1|_{\sB_i}\!: \sB_i \inj \Sym_\CC^{[1]} \Omega_{F_i}^{q}(\log E_i^c|_{F_i}),
\end{sequation}%
with respect to which we have $\kappa_\CC(\ms B_i) \le \dim F_i - 1$.
\end{clm}

\begin{proof}
We start by considering the inclusion $\sB \subset \Omega_Z^{[p]}(\log \rd D_Z.)$. First, we restrict to $E_i$. By \cite[Theorem 11.7.2]{GKKP11} and by Proposition \ref{prp:dlt adjunction}(\ref{itm:6.8.2}), we have the residue sequence
\[ 0 \to \Omega_{E_i}^{[p]}(\log \rd E_i^c.) \to \Omega_Z^{[p]}(\log \rd D_Z.)\big|_{E_i}^{**} \to \Omega_{E_i}^{[p-1]}(\log \rd E_i^c.) \to 0, \]
which is exact off a small subset of $E_i$. By Lemma \ref{lem:saturation-restriction}, the sheaf $\ms B|_{E_i}^{**}$ injects into the middle term of this sequence. Applying the Reduction Lemma \ref{lem:reduction} and using that $\Omega_{E_i}^{[r]}(\log \rd E_i^c.) = \Sym_\CC^{[1]} \Omega_{E_i}^r(\log E_i^c)$, we obtain \eqref{eqn:can emb}.

Next, we restrict to $F_i$. Applying Proposition \ref{prp:rel diff C-pairs} to the subsheaf in \eqref{eqn:can emb}, we obtain the embeddings $\alpha_k$ in \eqref{eqn:alpha_k}.

Finally, we prove the statement about the \CC-Kodaira dimension of $\sB_i$. By Lemma \ref{lem:mindltmodel 3}, the pair $(F_i, E_i^c|_{F_i})$ satisfies the assumptions of weak Bogomolov--Sommese vanishing, Proposition \ref{prp:BS van on dlt C-pairs}. An application of the latter to the embedding \eqref{eqn:alpha_1'} yields $\kappa_\CC(\ms B_i) \le \dim F_i - 1$, thus finishing the proof of Claim \ref{clm:notbig}.
\end{proof}

Recall from the discussion in Section~\ref{subsec:outline of prf} that for Weil divisorial sheaves, reflexive pull-back does not commute with taking reflexive tensor powers. The following claim tells us that in our situation, where pull-back is restriction to $F_i$ and tensor powers are replaced by $\Sym_\CC^{[k]} (\,\cdot\,)$, we have at least an inclusion one way.

\begin{clm}[Remedy for non-commutativity] \label{clm:subst}
For any number $k$, there is an injection
\begin{sequation} \label{seq:subst}
\bigl( (\Sym_\CC^{[k]} \sB)|_{E_i}^{**} \bigr)\big|_{F_i}^{**} \inj \Sym_\CC^{[k]} \sB_i,
\end{sequation}%
where $\Sym_\CC^{[k]} \sB_i$ is defined with respect to the embedding \eqref{eqn:alpha_1'}.
\end{clm}

\begin{proof}
Let $k$ be any number. By Theorem \ref{thm:res of symm diff}, there is a map
\begin{sequation} \label{eqn:1421}
(\Sym_\CC^{[k]} \sB)\big|_{E_i}^{**} \to \Sym_\CC^{[k]} \Omega_{E_i}^r(\log E_i^c)
\end{sequation}%
for the same $r$ as in \eqref{eqn:can emb}. We use Theorem \ref{thm:res of symm diff}(\ref{itm:6.12.1}) or Theorem \ref{thm:res of symm diff}(\ref{itm:6.12.2}), according to whether $r = p - 1$ or $r = p$.
As explained in Theorem \ref{thm:res of symm diff}, the map \eqref{eqn:1421} generically coincides with the $k$-th power of the embedding \eqref{eqn:can emb}, hence it is injective too.
By Lemma \ref{lem:gen eq} there is an inclusion
\begin{sequation} \label{eqn:817}
(\Sym_\CC^{[k]} \sB)\big|_{E_i}^{**} \subset \Sym_\CC^{[k]} (\sB|_{E_i}^{**}),
\end{sequation}%
because the sheaves generically agree as subsheaves of $\Sym_\CC^{[k]} \Omega_{E_i}^r(\log E_i^c)$, and the latter is saturated. Restricting this inclusion to $F_i$, we obtain
\[ \bigl( (\Sym_\CC^{[k]} \sB)|_{E_i}^{**} \bigr)\big|_{F_i}^{**} \subset \bigl( \Sym_\CC^{[k]} (\sB|_{E_i}^{**}) \bigr)\big|_{F_i}^{**}. \]

The embedding \eqref{eqn:alpha_1'} induces a saturated subsheaf
\[ \Sym_\CC^{[k]} \sB_i \subset \Sym_\CC^{[k]} \Omega_{F_i}^q(\log E_i^c|_{F_i}). \]
In order to prove Claim \ref{clm:subst}, it suffices to show that the image of $\alpha_k$ is contained in this subsheaf, because then the injection \eqref{seq:subst} is given as
\[ \bigl( (\Sym_\CC^{[k]} \sB)|_{E_i}^{**} \bigr)\big|_{F_i}^{**} \subset \bigl( \Sym_\CC^{[k]} (\sB|_{E_i}^{**}) \bigr)\big|_{F_i}^{**} \stackrel{\alpha_k}{\lhook\joinrel\longrightarrow} \Sym_\CC^{[k]} \sB_i. \]
But since the maps $\alpha_k$ satisfy the compatibility condition (\ref{prp:rel diff C-pairs}.\ref{itm:cp2}) and $\sB_i$ is the image of $\alpha_1'$, Lemma \ref{lem:gen eq} applies to show that $\img \alpha_k$ is contained in $\Sym_\CC^{[k]} \sB_i$.
\end{proof}

\subsection{Step 3: Pulling back sections}

Recall that we need to show Claim \ref{clm:1060}. We start with an observation.

\begin{obs}
Fix any number $k > 0$. Under the isomorphism induced by $f$, we have $\bigl(\Sym_\CC^{[k]} \sA\bigr)\big|_{X \minus f(\Exc f)} \isom \bigl(\Sym_\CC^{[k]} \sB\bigr)\big|_{Z \minus \Exc(f)}$. Hence there is an injective ``push-forward'' map
\begin{sequation} \label{eqn:6.3}
f_*\!: H^0(Z, \Sym_\CC^{[k]} \sB) \to H^0(X, \Sym_\CC^{[k]} \sA),
\end{sequation}%
given by the composition
\begin{align*}
H^0(Z, \Sym_\CC^{[k]} \sB) &\xrightarrow{\;\restr\;}                       \Hnought(Z \minus \Exc(f), \Sym_\CC^{[k]} \sB) \\
                         &\xrightarrow{\hspace{.724em}\sim\hspace{.724em}} \Hnought(X \minus f(\Exc f), \Sym_\CC^{[k]} \sA) \\
                         &\xrightarrow{\hspace{.724em}\sim\hspace{.724em}} H^0(X, \Sym_\CC^{[k]} \sA).
\end{align*}
For the last map, we used that $f(\Exc f)$ has codimension $\ge 2$ in $X$ and $\Sym_\CC^{[k]} \sA$ is reflexive.
\end{obs}

To prove Claim \ref{clm:1060}, it suffices to show the surjectivity of $f_*$, for all values of $k$.
To this end, let $\sigma \in H^0(X, \Sym_\CC^{[k]} \sA)$ be any section.
We may regard $f^*(\sigma) \in H^0\bigl(Z, f^{[*]}(\Sym_\CC^{[k]} \sA)\bigr)$ as a rational section of $\Sym_\CC^{[k]} \sB$, possibly with poles along~$E$.
\begin{clm} \label{clm:no poles}
In fact, there are no such poles, i.e.~$f^*(\sigma) \in H^0(Z, \Sym_\CC^{[k]} \sB)$.
\end{clm}
Assuming Claim \ref{clm:no poles}, $f^*(\sigma)$ is an $f_*$-preimage of $\sigma$.~--- To prove Claim \ref{clm:no poles}, let $G$ be the pole divisor of $f^*(\sigma)$ as a rational section of $\Sym_\CC^{[k]} \sB$, i.e.~the minimal effective divisor on $Z$ such that
\[ f^*(\sigma) \in H^0 \bigl( Z, \bigl(\Sym_\CC^{[k]} \sB \tensor \mc O_Z(G)\bigr)^{**} \bigr). \]
Clearly $G$ is supported on $E$, and we aim to show that $G = 0$.

We proceed by contradiction. Suppose that $G \ne 0$. Then by the Negativity Lemma \ref{prp:negativity}, there is an index $1 \le i \le \ell$ such that $-G|_{F_i}$ is big. By Remark \ref{rem:passing to a power}, passing to a suitable power of $\sigma$ we may assume that $G$ is Cartier. Then $\O_{F_i}(-G|_{F_i})$ is a big line bundle. Note that $f^*(\sigma)$, as a section of
\[ \Sym_\CC^{[k]} \sB \tensor \mc O_Z(G) \isom \sHom \bigl( \O_Z(-G), \Sym_\CC^{[k]} \sB \bigr), \]
does not vanish along $F_i$, by the minimality of $G$. Restricting to $E_i$ and then to $F_i$, we get a morphism
\[ \O_{F_i}(-G|_{F_i}) = \bigl( \O_Z(-G)|_{E_i} \bigr)\big|_{F_i} \to \bigl( (\Sym_\CC^{[k]} \sB)|_{E_i}^{**} \bigr)\big|_{F_i}^{**} \]
which is nonzero, hence injective. So $\bigl( (\Sym_\CC^{[k]} \sB)|_{E_i}^{**} \bigr)\big|_{F_i}^{**}$ is a big Weil divisorial sheaf. Now by Claim \ref{clm:subst}, this shows that $\Sym_\CC^{[k]} \ms B_i$ is big too. But this contradicts the non-bigness Claim \ref{clm:notbig}. We conclude that $G = 0$, which proves Claim \ref{clm:no poles}. \qed

\section{Sharpness of Theorem \ref*{thm:lc BS van}} \label{sec:sharpness}

\PreprintAndPublication{
In this section, we show that the assumptions of our main result, Theorem \ref{thm:lc BS van}, cannot be weakened. To be more precise, in Example \ref{exm:DB no BS van} we will exhibit normal projective \Q-Gorenstein threefolds $X$ that have \emph{Du Bois} singularities (but are not log canonical), such that $\Omega_X^{[2]}$ admits a \Q-ample \Q-Cartier subsheaf.

For an introduction to Du Bois singularities, see \cite{KS09}. Note that for $X$ to be Du Bois, it is not required that $K_X$ be \Q-Cartier. It is a basic fact that the underlying space of a log canonical pair has Du Bois singularities, cf.~\cite[Thm.~1.4]{KK10}. So Du Bois singularities can be seen as a generalization of log canonicity to the non-\Q-Gorenstein case. However, a \Q-Gorenstein Du Bois singularity is not necessarily log canonical.

Since the proof of the Extension Theorem \cite[Thm.~16.1]{GKKP11} heavily relies on vanishing theorems for spaces with Du Bois singularities, it is natural to ask whether Theorem~\ref{thm:lc BS van} also carries over to this more general setting. The following example shows that this is not the case.
Philosophically speaking, this tells us that Theorem~\ref{thm:lc BS van} is not a consequence of the closedness of holomorphic forms alone.

\begin{exm} \label{exm:DB no BS van}
Let $S$ be a smooth projective surface such that $K_S$ is ample, but the Hodge numbers $h^{0,1}(S)$ and $h^{0,2}(S)$ are zero. For example, let $S$ be any Catanese surface \cite{Cat81}. Choose a number $n \gg 0$ such that $nK_S$ is very ample and defines a projectively normal embedding $S \inj \P^N$. Let $X \subset \P^{N+1}$ be the projective cone over $S$ with respect to this embedding. Then $X$ is normal, because the embedding of $S$ is projectively normal. Moreover, using \cite[Ex.~3.5]{HK10} it can be checked that $K_X$ is \Q-Cartier.

Next, we show that $X$ has Du Bois singularities. For this, let $f\!: \tilde X \to X$ be the blowup of the vertex $P \in X$. The map $f$ is a resolution of $X$, and the sheaves $R^i f_* \O_{\tilde X}$, $i > 0$, are skyscraper sheaves supported at $P$. A Leray spectral sequence computation shows that the stalk $(R^i f_* \O_{\tilde X})_P$ is equal to
\[ (R^i f_* \O_{\tilde X})_P = \bigoplus_{m \ge 0} H^i(S, mnK_S) \quad \text{for all $i > 0$}. \]
By Kodaira vanishing and the assumption on the Hodge numbers of $S$, all the summands vanish, hence the natural map $\O_X \to Rf_* \O_{\tilde X}$ is a quasi-isomorphism, and $X$ has rational singularities. In particular, $X$ is Du Bois \cite[Cor.~2.6]{Kov99}. On the other hand, an elementary discrepancy calculation shows that $(X, \emptyset)$ is not log canonical.

In order to see that $X$ fails Bogomolov--Sommese vanishing, let $\iota\!: X \minus \{ P \} \inj X$ be the inclusion of the smooth locus, and let $p\!: X \minus \{ P \} \to S$ be the natural projection map. Then
\[ \sA := \iota_*(p^* \omega_S) \subset \Omega_X^{[2]} \]
is a \Q-ample \Q-Cartier subsheaf, because its $n$-th reflexive power $\sA^{[n]}$ is exactly the hyperplane bundle associated to the given embedding $X \subset \P^{N+1}$. This ends the construction of our counterexample.
\end{exm}

In Example \ref{exm:DB no BS van}, it is important that the sheaf $\sA$ is not invertible: A close examination of the proof of \cite[Thm.~16.1]{GKKP11} yields the following statement, which we state without proof for completeness' sake.

\begin{thm}[Bogomolov--Sommese vanishing for invertible sheaves on Du Bois spaces] \label{thm:DB BS van Cartier}
Let $X$ be a normal projective variety with Du Bois singularities. If $\sA \subset \Omega_X^{[p]}$ is an invertible subsheaf, then $\kappa(\ms A) \le p$. \qed
\end{thm}

Many of the proofs presented in \cite{GKKP11} rely on taking index one covers. The reason that Theorem \ref{thm:DB BS van Cartier} fails for \Q-Cartier sheaves is that the Du Bois property is not preserved by taking finite covers \'etale in codimension one.
}{
In this section, we show that the assumptions of our main result, Theorem \ref{thm:lc BS van}, cannot be weakened. To be more precise, in Example \ref{exm:DB no BS van} we will exhibit normal projective \Q-Gorenstein threefolds $X$ that have Du Bois singularities (but are not log canonical), such that $\Omega_X^{[2]}$ admits a \Q-ample \Q-Cartier subsheaf.

Since the proof of the Extension Theorem \cite[Thm.~16.1]{GKKP11} heavily relies on vanishing theorems for spaces with Du Bois singularities, it is natural to ask whether Theorem~\ref{thm:lc BS van} also carries over to this more general setting. The following example shows that this is not the case.
Philosophically speaking, this tells us that Theorem~\ref{thm:lc BS van} is not a consequence of the closedness of holomorphic forms alone.

\begin{exm} \label{exm:DB no BS van}
Let $S$ be a smooth projective surface such that $K_S$ is ample, but the Hodge numbers $h^{0,1}(S)$ and $h^{0,2}(S)$ are zero. For example, let $S$ be any Catanese surface \cite{Cat81}. Choose a number $n \gg 0$ such that $nK_S$ is very ample and defines a projectively normal embedding $S \inj \P^N$. Let $X \subset \P^{N+1}$ be the projective cone over $S$ with respect to this embedding, with vertex $P \in X$. Then $X$ is normal, because the embedding of $S$ is projectively normal. It is not difficult to check that $K_X$ is \Q-Cartier.
Moreover, a Leray spectral sequence computation, using Kodaira vanishing and the assumption on the Hodge numbers of $S$, shows that $X$ has rational singularities. In particular, $X$ is Du Bois \cite[Cor.~2.6]{Kov99}. On the other hand, an elementary discrepancy calculation shows that $(X, \emptyset)$ is not log canonical.

In order to see that $X$ fails Bogomolov--Sommese vanishing, let $\iota\!: X \minus \{ P \} \inj X$ be the inclusion of the smooth locus, and let $p\!: X \minus \{ P \} \to S$ be the natural projection map. Then
\[ \sA := \iota_*(p^* \omega_S) \subset \Omega_X^{[2]} \]
is a \Q-ample \Q-Cartier subsheaf, because its $n$-th reflexive power $\sA^{[n]}$ is exactly the hyperplane bundle associated to the given embedding $X \subset \P^{N+1}$. This ends the construction of our counterexample.
\end{exm}
}

\PreprintAndPublication{
\part{COROLLARIES} \label{part:IV}
}{
\part{A COROLLARY} \label{part:IV}
}

\PreprintAndPublication{
\section{A remark on a conjecture of Campana} \label{sec:rem on kappa++}

In \cite{Cam10}, Campana made the following definition and conjecture.

\begin{dfn}[Augmented Kodaira dimension, see {\cite[Def.~2.6]{Cam10}}]
Let $X$ be a projective manifold. Then the \emph{augmented Kodaira dimension} of $X$ is
\[ \kappa_{++}(X) := \max \bigl\{ \kappa(\sL) \;\big|\; \sL \text{ an invertible subsheaf of $\Omega_X^p$, for some $p > 0$} \bigr\}. \]
\end{dfn}

Obviously, $\kappa(X) := \kappa(K_X) \le \kappa_{++}(X) \le \dim X$.

\begin{conj}[see {\cite[Conj.~2.7]{Cam10}}] \label{conj:kappa++}
Let $X$ be a projective manifold with $\kappa(X) \ge 0$. Then $\kappa_{++}(X) = \kappa(X)$.
\end{conj}

It follows immediately from classical Bogomolov--Sommese vanishing (Theorem \ref{thm:BS van}) that if $\kappa_{++}(X) = \dim X$, then $X$ is of general type. This is a special case of Conjecture \ref{conj:kappa++}. Here we would like to remark that by the same argument, Theorem \ref{thm:lc BS van} leads to the following corollary.

\begin{cor} \label{cor:easy lc kappa++}
Let $(X, D)$ be a projective log canonical pair. Set
\begin{align*}
\kappa_{++}(X, D) := \max \bigl\{ \kappa(\sA) \;\big|\; \sA \text{ a Weil divisorial subsheaf of $\Omega_X^{[p]}(\log \rd D.)$, \phantom{.}} & \\
\text{for some $p > 0$} \bigr\}. &
\end{align*}
If $\kappa_{++}(X, D) = \dim X$, then $(X, D)$ is of log general type. \qed
\end{cor}
}{}

\section{A Kodaira--Akizuki--Nakano-type vanishing result} \label{sec:Serre dual}

In \cite[Proposition 4.5.2]{GKP12}, the authors obtained a Kodaira--Akizuki--Nakano-type vanishing result for reflexive differentials twisted by an ample line bundle. Arguing along the same lines, we show that this vanishing also holds for twists by Weil divisorial sheaves.
\PreprintAndPublication{
Furthermore, we observe that the sheaves of reflexive differentials may be replaced by K\"ahler differentials.
}{
Furthermore, we observe that the sheaves of reflexive differentials may be replaced by K\"ahler differentials for dimension reasons.
}

\begin{cor}[KAN-type vanishing for top cohomology] \label{cor:Serre dual lc BS van}
Let $(X, D)$ be a complex projective log canonical pair of dimension $n$ and $\sA$ a Weil divisorial sheaf on $X$. Then
\[ H^n \big( X, (\Omega_X^{[p]}(\log \rd D.) \tensor \sA)^{**} \big) = 0 \]
for $p > n - \kappa(\sA)$.
\end{cor}

\begin{proof}
By Theorem \ref{thm:lc BS van},
\[ \Hom \big( \sA, \Omega_X^{[q]}(\log \rd D.) \big) = 0 \]
for $q < \kappa(\sA)$. So
\begin{align*}
H^n \big( X, (\Omega_X^{[p]}(\log \rd D.) \tensor \sA)^{**} \big)^* &\isom
\Hom \big( (\Omega_X^{[p]}(\log \rd D.) \tensor \sA)^{**}, \omega_X \big) \\
&\subset \Hom \big( (\Omega_X^{[p]}(\log \rd D.) \tensor \sA)^{**}, \Omega_X^{[n]}(\log \rd D.) \big) \\
&\isom \Hom \big( \sA, \Omega_X^{[n-p]}(\log \rd D.) \big) \\
&= 0.
\end{align*}
The first isomorphism exists because $\omega_X$ is the dualizing sheaf \cite[Def.~5.66 and Prop.~5.75]{KM98}. The second isomorphism comes from the pairing associated to the wedge product of differential forms.
\end{proof}

\PreprintAndPublication{
For dimension reasons, Corollary \ref{cor:Serre dual lc BS van} remains true if we drop all the double duals.
}{}

\begin{cor}[KAN-type vanishing for K\"ahler differentials] \label{cor:Serre dual lc BS van II}
Let $(X, D)$ be a complex projective log canonical pair of dimension $n$ and $\sA$ a Weil divisorial sheaf on $X$. Then
\[ H^n \big( X, \Omega_X^p(\log \rd D.) \tensor \sA \big) = 0 \]
for $p > n - \kappa(\sA)$.
\end{cor}

\PreprintAndPublication{
\begin{proof}
We have a natural map
\[ \alpha\!: \Omega_X^p(\log \rd D.) \tensor \sA \to \big(\Omega_X^{[p]}(\log \rd D.) \tensor \sA\big)^{**}, \]
defined by the composition
\[ \xymatrix{
\Omega_X^p(\log \rd D.) \tensor \sA \ar[r]^-{r_1 \tensor \id} & \Omega_X^{[p]}(\log \rd D.) \tensor \sA \ar[r]^-{r_2} & \big(\Omega_X^{[p]}(\log \rd D.) \tensor \sA\big)^{**}.
} \]
Here $r_1$ and $r_2$ denote the maps to the double duals of the appropriate sheaves.

The map $\alpha$ fits into the four-term exact sequence
\begin{sequation} \label{seq:4-term}
0 \to \sT \to \Omega_X^p(\log \rd D.) \tensor \sA \stackrel\alpha\to \big(\Omega_X^{[p]}(\log \rd D.) \tensor \sA\big)^{**} \to \mc Q \to 0,
\end{sequation}%
where $\sT := \ker(\alpha)$ and $\sQ := \coker(\alpha)$. Since $\alpha$ is an isomorphism on the snc locus of $(X, \rd D.)$, the sheaves $\sT$ and $\mc Q$ are supported on a set of codimension $\ge 2$. In particular,
\begin{sequation} \label{seq:T, Q van}
\phantom{\text{$i \ge n - 1$}}
H^i(X, \sT) = H^i(X, \mc Q) = 0 \text{\qquad for $i \ge n - 1$}
\end{sequation}%
by dimension reasons. Now chop up sequence (\ref{seq:4-term}) into the short exact sequences
\begin{sequation} \label{seq:ses 1}
0 \to \sT \to \Omega_X^p(\log \rd D.) \tensor \sA \stackrel\alpha\to \img(\alpha) \to 0
\end{sequation}%
and
\begin{sequation} \label{seq:ses 2}
0 \to \img(\alpha) \to \big(\Omega_X^{[p]}(\log \rd D.) \tensor \sA\big)^{**} \to \mc Q \to 0.
\end{sequation}%
By the long exact sequence associated to (\ref{seq:ses 2}),
\[ \underbrace{H^{n-1}(X, \mc Q)}_{\text{$= 0$ by (\ref{seq:T, Q van})}} \to H^n\big(X, \img(\alpha)\big) \to \underbrace{H^n \big(X, (\Omega_X^{[p]}(\log \rd D.) \tensor \sA)^{**} \big)}_{\text{$= 0$ by Corollary \ref{cor:Serre dual lc BS van}}}, \]
we get $H^n\big(X, \img(\alpha)\big) = 0$. Then by the long exact sequence associated to (\ref{seq:ses 1}),
\[ \underbrace{H^n(X, \sT)}_{\text{$= 0$ by (\ref{seq:T, Q van})}} \to H^n\big(X, \Omega_X^p(\log \rd D.) \tensor \sA\big) \to \underbrace{H^n\big(X, \img(\alpha)\big)}_{\text{$= 0$}}, \]
we get $H^n\bigl(X, \Omega_X^p(\log \rd D.) \tensor \sA\bigr) = 0$, as desired.
\end{proof}
}{
\begin{proof}
We have a natural map
\[ \alpha\!: \Omega_X^p(\log \rd D.) \tensor \sA \to \big(\Omega_X^{[p]}(\log \rd D.) \tensor \sA\big)^{**}, \]
defined by the composition
\[ \xymatrix{
\Omega_X^p(\log \rd D.) \tensor \sA \ar[r]^-{r_1 \tensor \id} & \Omega_X^{[p]}(\log \rd D.) \tensor \sA \ar[r]^-{r_2} & \big(\Omega_X^{[p]}(\log \rd D.) \tensor \sA\big)^{**}.
} \]
Here $r_1$ and $r_2$ denote the maps to the double duals of the appropriate sheaves.

The map $\alpha$ fits into the four-term exact sequence
\begin{sequation} \label{seq:4-term}
0 \to \sT \to \Omega_X^p(\log \rd D.) \tensor \sA \stackrel\alpha\to \big(\Omega_X^{[p]}(\log \rd D.) \tensor \sA\big)^{**} \to \mc Q \to 0,
\end{sequation}%
where $\sT := \ker(\alpha)$ and $\sQ := \coker(\alpha)$. Since $\alpha$ is an isomorphism on the snc locus of $(X, \rd D.)$, the sheaves $\sT$ and $\mc Q$ are supported on a set of codimension $\ge 2$. In particular,
\begin{sequation} \label{seq:T, Q van}
\phantom{\text{$i \ge n - 1$}}
H^i(X, \sT) = H^i(X, \mc Q) = 0 \text{\qquad for $i \ge n - 1$}
\end{sequation}%
by dimension reasons. Chopping up sequence (\ref{seq:4-term}) into two short exact sequences and looking at the associated long exact sequences, \eqref{seq:T, Q van} and Corollary \ref{cor:Serre dual lc BS van} yield the desired vanishing.
\end{proof}
}

\newcommand{\etalchar}[1]{$^{#1}$}

\end{document}